\documentclass[a4paper,onesided,11pt]{article}
\usepackage[DIV=11]{typearea}
\usepackage{amsmath,amssymb}
\usepackage{amsthm}
\usepackage{mathtools}
\usepackage{a4wide}
\usepackage[dvipsnames]{xcolor}
\usepackage[unicode, final, colorlinks=true]{hyperref}
\usepackage{textgreek}
\usepackage{enumitem, booktabs}
\usepackage{xcolor}
\usepackage{siunitx}
\usepackage{bbm}

\hypersetup{citecolor=teal, linkcolor=BrickRed, urlcolor=NavyBlue}

\newcommand{\N}{\mathbb{N}}
\newcommand{\R}{\mathbb{R}}

\newcommand{\ddt}{\partial_t}

\newcommand{\bU}{\mathbf{U}}
\newcommand{\bV}{\mathbf{V}}
\newcommand{\bQ}{\mathbf{Q}}

\newcommand{\bF}{\mathbf{F}}
\newcommand{\bT}{\mathbf{T}}
\newcommand{\bA}{\mathbf{A}}
\newcommand{\bw}{\mathbf{w}}
\newcommand{\bI}{\mathbf{I}}
\newcommand{\bSigma}{\mathbf{\Sigma}}
\newcommand{\bLambda}{\mathbf{\Lambda}}
\newcommand{\bzero}{\mathbf{0}}
\newcommand{\bx}{\mathbf{x}}
\newcommand{\bn}{\mathbf{n}}
\newcommand{\bt}{\mathbf{t}}
\newcommand{\bj}{\mathbf{j}}
\newcommand{\bv}{\mathbf{v}}
\newcommand{\br}{\mathbf{r}}
\newcommand{\pdiff}[2]{\frac{\partial\, #1}{\partial\, #2}}

\newcommand{\rs}{\mathcal{RS}}

\newcommand{\CFL}{\text{CFL}}

\newtheorem{definition}{Definition}[section]
\newtheorem{theorem}{Theorem}[section]
\newtheorem{lemma}[theorem]{Lemma}
\newtheorem{proposition}[theorem]{Proposition}
\newtheorem{remark}[theorem]{Remark}
\newtheorem{algorithm}[theorem]{Algorithm}

\providecommand{\msc}[1]{\textit{2020 MSC:} #1}

\begin{document}
\date{
  \small
  $^1$Institute of Geometry and Applied Mathematics,\\ RWTH Aachen University, Im Süsterfeld 2,\\ 52072 Aachen, Germany\\
  \smallskip
  $^2$Chair of Geometry and Analysis,\\ RWTH Aachen University, Kreuzherrenstraße 2,\\ 52062 Aachen, Germany\\
  \bigskip
   January 2026
  }
\author{Niklas Kolbe$^{1,}$\footnote{Corresponding author. Address: Institute of Geometry and Applied Mathematics, RWTH Aachen University, Templergraben 55, 52062 Aachen, Germany. E-mail: \tt{kolbe@igpm.rwth-aachen.de}} \and Siegfried Müller$^1$ \and Aleksey Sikstel$^2$}

\title{Discontinuous Galerkin schemes for multi-dimensional coupled hyperbolic systems}
\maketitle

\begin{abstract}
 A novel class of Runge-Kutta discontinuous Galerkin schemes for coupled systems of conservation laws in multiple space dimensions that are separated by a fixed sharp interface is introduced. The schemes are derived from a relaxation approach and a local projection and do not require expensive solutions of nonlinear half-Riemann problems. The underlying Jin-Xin relaxation involves a problem specific modification of the coupling condition at the interface, for which a simple construction algorithm is presented. The schemes are endowed with higher order time discretization by means of strong stability preserving Runge-Kutta methods. These are derived from an asymptotic preserving implicit-explicit treatment of the coupled relaxation system taken to the discrete relaxation limit. In a case study the application to a multi-dimensional fluid-structure coupling problem employing the compressible Euler equations and a linear elastic model is discussed.
  
  \bigskip
  \noindent \msc{35L60, 35R02, 65M60, 65N55, 74F10}
  \end{abstract}
%\tableofcontents
\section{Introduction}
Many phenomena in the natural sciences and engineering are governed by interacting physical processes that are most naturally modeled in terms of coupled systems of partial differential equations (PDEs). In such settings, different PDE models apply in different spatial regions and their mutual influence is expressed through coupling conditions at the interfaces. Prominent examples include the design of composite materials, where distinct thermo-mechanical behaviors coexist in a heterogeneous structure~\cite{ciarletta1993noncl, ghasemi2018}, the study of multiphase flows, in which immiscible fluids with different properties are separated by evolving interfaces~\cite{abgrall2001comput, ndanou2015multi, tryggvason2001} and fluid–structure interaction problems, where fluid flow models are coupled to deformable structural models along a shared boundary~\cite{takizawa2011multis, farhat2012fiver}. The study of these and related applications require efficient and reliable numerical methods capable of solving coupled PDE systems.

The case of coupled nonlinear hyperbolic systems is particularly relevant for the applications mentioned above and poses substantial mathematical and numerical challenges even for static interfaces. Iterative approaches have been frequently employed to discretize such problems, typically by alternating between solvers adressing the different subdomains and enforcing the coupling conditions in an outer iteration~\cite{hou2012numerreview}. While these schemes are common in fluid–structure interaction with moving boundaries, the repeated use of the solvers and stringent stability requirements make them computationally expensive. In addition, the convergence behaviour of iterative schemes is yet unknown and difficult to analyze. An alternative is to adopt an interface treatment more closely aligned with Godunov’s method and front-tracking techniques~\cite{bressan2000hyper}, on which most modern schemes for hyperbolic systems—such as finite volume methods~\cite{harten1983upstrdiffergodun}, MUSCL and WENO reconstructions~\cite{leer1979towarultimconserdifferschem, shu1998essen} and discontinuous Galerkin (DG) methods~\cite{cockburn1998rungekuttagaler}—are based. In this framework, two coupled half-Riemann problems are posed at the interface; their solutions define intermediate coupling states that encode the coupling conditions and yield interface fluxes tailored to the distinct systems on either side, see e.g.~\cite{godlewski2004numerintercoupl, godlewski2005numericinter, chalons2008, herty2019couplcompreulerequat}. This construction relies on the Lax curves of the coupled systems~\cite{lax1973hyper}, which may not be available in closed form~\cite{hantke2018analysimulnew} and, even when they are, often lead to highly nonlinear algebraic systems that can only be solved approximately. To mitigate this difficulty, several methods have been recently proposed that simplify the process by either resorting to linearization around suitable reference states or by employing approximate Riemann solvers~\cite{toro2009rieman}, thereby reducing the complexity while retaining the main structural features of the hyperbolic coupling. In~\cite{banda2015} a linearization technique for the coupling conditions in case of gas-flow has been proposed. Another linearization approach specific to subsonic flows based on an implicit formulation has been developed in~\cite{kolb2010}. 

Closely related concepts arise in the study of hyperbolic flow dynamics on networks, where, at each node, several incoming and outgoing edges meet and more than two half-Riemann problems are coupled simultaneously, see, for example,~\cite{bressan2014flowsnetwor} for representative results in this direction. The simplest case concerns conservation laws with a single spatial discontinuity in the flux function, for which, in the nonlinear one-dimensional setting, a classification theory of admissible interface conditions and corresponding solution concepts has been developed in~\cite{adimurthi2005optim, andreianov2011}.

A large part of the literature on coupled hyperbolic problems relies on application-specific treatments, is limited to one spatial dimension or does not easily allow for higher order approximations, e.g.,~\cite{colombo2010coupleuler, herty2003model}. Extending one-dimensional approaches to genuinely multi-dimensional domains introduces additional geometric complications, since interfaces are no longer isolated points but $(d-1)$-dimensional surfaces whose orientation and location relative to the computational mesh must be taken into account. This can be addressed either by solving fully multidimensional coupled Riemann problems at the interface~\cite{balsara2012hllcrieman} or by employing suitable projections onto local normal directions, leading to families of quasi one-dimensional coupling problems. The latter strategy has recently been followed in~\cite{herty2018fluid} to couple compressible flows to structure models on a two-dimensional domain. While some higher order methods have been developped, see e.g.~\cite{banda2016numer, borsche2014aderschemhigh, du2026couplriemaneuler}, many existing techniques are embedded in a fixed numerical framework and do not admit systematic extensions to higher-order discretizations in time and space. This limitation is particularly apparent for finite volume methods, which often remain first-order accurate in space near the coupling interface~\cite{godlewski2005}, because higher-order reconstructions across interfaces—requiring consistent data from both sides and solutions to generalized half-Riemann problems—are difficult to construct.

In a recent series of works, a relaxation-based coupling strategy for hyperbolic systems has been developed~\cite{herty2023centr, herty2023centrschemtwo, kolbe2024numerschemcoupl}. The key idea has been to embed the original problem into the extended but linearised Jin--Xin relaxation system~\cite{jin1995relaxschemsystem}, which enables a treatment of the coupling interface not relying on Lax curves. Combined with an asymptotic-preserving discretization, this framework yields, in the relaxation limit, a simple central scheme for the original coupled problem. So far, the methodology has mainly been applied in one spatial dimension with first-order accurate interface approximations, yet it has already led to new numerical schemes for networked blood flow~\cite{beckers2025laxfried} and for two-phase fluid–structure interaction~\cite{kolbe2025relax}. In this context, different relaxation systems have also been used both to derive appropriate coupling conditions, see~\cite{borsche2018kineti, zhou2020consti}, and to discretize conservation laws with discontinuous flux functions~\cite{karlsen2003relaxschemconser}. A related direction is the construction of vanishing-viscosity solutions, as pursued in~\cite{egger2021, towers2022explicfinitvolum}, where schemes formally tailored to parabolic systems aim to capture the hyperbolic coupling in the small-viscosity limit.

In this work we propose a numerical framework for general multi-dimensional coupled hyperbolic systems generalizing the relaxation-based approach. In the multi-dimensional setting, the interface treatment of the relaxed problem is based on local projections at the interface points. As additional coupling conditions imposed on the relaxation variables are required for an adequate closure~\cite{cao2022constii} we provide a general recipe for their construction. On this basis, we develop a class of modal discontinuous Galerkin schemes of arbitrary polynomial order in space for the coupled problem and show that they arise naturally as the relaxation limit of a Godunov-type DG scheme for the coupled relaxation system. In analogy, we demonstrate that high-order strong-stability-preserving (SSP) time discretizations, see~\cite{gottlieb2001stronstabilpreser}, can be obtained by taking the relaxation limit in appropriately designed asymptotic-preserving implicit–explicit (IMEX) Runge–Kutta schemes.

The rest of the paper is organized as follows. In Section~\ref{sec:problemrelaxation}, we formally introduce the nonlinear hyperbolic coupling problem, its global relaxation and the quasi-one-dimensional localized relaxation system. The latter requires an adapted coupling condition, for which consistency with the original coupling is defined in Section~\ref{sec:relaxedcoupling}. Section~\ref{sec:riemann} discusses the corresponding coupled Riemann problem at the interface. Here we provide an explicit procedure in Algorithm~\ref{algo:rsconstruction} for the construction of a problem specific Riemann solver. In Section~\ref{sec:DG}, introducing a DG scheme for the coupled relaxation system with IMEX time discretization we derive a high order SSP-DG scheme for the original problem with final form provided in Section~\ref{sec:relaxationlimit}. Section~\ref{sec:example} is devoted to an example application of our approach to the fluid-structure coupling problem studied in~\cite{herty2018fluid}. Herein we derive an explicit Riemann solver for the 2D coupling problem and verify the efficiency of our approach in an numerical experiment, after which the paper is concluded in Section~\ref{sec:conclusion}.

\section{The coupled problem and localized relaxation systems}\label{sec:problemrelaxation}
In this section we first introduce the central model problem of this work in Section~\ref{sec:modelproblem} then apply global Xin-Jin relaxation in Section~\ref{sec:globalrelaxation} and consider a projection at an interface point in Section~\ref{sec:locrelaxationsystem}. In Section~\ref{sec:relaxedcoupling} we relate the coupling condition of the model problem to the one of the localized relaxation system.
\subsection{The model problem}\label{sec:modelproblem}
This work addresses coupled systems in the space $\R^d$ for $d\geq2$. Let $\Gamma \subset \R^d$ denote a smooth manifold of dimension $d-1$ that separates the space into two disjoint open subsets $\Omega_1$ and $\Omega_2$, i.e., it holds $\Omega_1 \sqcup \Omega_2 \sqcup \Gamma= \R^d$ and $\partial \Omega_1 \cap \partial \Omega_2 = \Gamma$. We consider the system
\begin{equation}\label{eq:system}
  \left\{
\begin{aligned}
  \frac{\partial \bU}{\partial t} + \sum_{j=1}^d \frac{\partial \bF_1^j(\bU)}{\partial x_j} &= \bzero \qquad \text{for }(t, \bx) \in (0, \infty) \times \Omega_1,\\
  \frac{\partial \bU}{\partial t} + \sum_{j=1}^d \frac{\partial \bF_2^j(\bU)}{\partial x_j} &= \bzero \qquad \text{for }(t, \bx) \in (0, \infty) \times \Omega_2
\end{aligned}
\right.
\end{equation}
with state variable $\bU$ given such that $ \bU(t, \bx) \in \mathcal{D}_1 \subseteq  \R^{m_1}$ if $\bx \in \Omega_1$ and $ \bU(t, \bx) \in \mathcal{D}_2 \subseteq  \R^{m_2}$ if $\bx \in \Omega_2$. The flux functions $\bF_i^j:\mathcal{D}_i \rightarrow \R^{m_i}$ for $j\in\{1,\dots, d\}$ and $i \in\{1,2\}$ are assumed to be continuously differentiable.
To connect both systems at the interface we impose the \emph{coupling condition}
\begin{equation}\label{eq:couplingu}
  \Psi_U^\bn(\bU(t, \bx^-) , \bU(t, \bx^+)) = \bzero \qquad \text{for a.e. } t \geq \bzero \quad \text{and }\bx \in \Gamma,
\end{equation}
where $\Psi_U^\bn: \mathcal{D}_1 \times \mathcal{D}_2 \to \R^\ell$ for some $\ell \in \N$ that depends on the problem at hand. By $\bn = \bn(\bx)$ we denote the unit normal of $\Gamma$ in $\bx$ that by convention points into $\Omega_2$ in this section. Note that through the normal vector the map in \eqref{eq:couplingu} depends on $\bx \in \Gamma$. Moreover, we use the notation $\bU(t, \bx^-)$ to refer to the left-sided limit (i.e.\, involving sequences $(\bx_n)_{n \in \N}\subset \Omega_1$) at time $t$ in direction $\bn$ and similarly the notation $\bU(t, \bx^+)$ to refer to the right-sided limit (i.e.\, involving sequences $(\bx_n)_{n \in \N}\subset \Omega_2$) at time $t$ in direction $\bn$.

We are interested in the Cauchy problem with respect to \eqref{eq:system} and \eqref{eq:couplingu} with initial data $\bU^0$ being in $L^\infty(\Omega_1)^{m_1}$ if restricted to $\Omega_1$ and in $L^\infty(\Omega_2)^{m_2}$ if restricted to $\Omega_2$ and such that
\[
  \Psi_U^\bn(\bU^0(\bx^-) , \bU^0(\bx^+)) = \bzero \qquad \forall \bx \in \Gamma
  \]
  holds for compatibility with the coupling condition.

\begin{figure}
  \centering
  \includegraphics{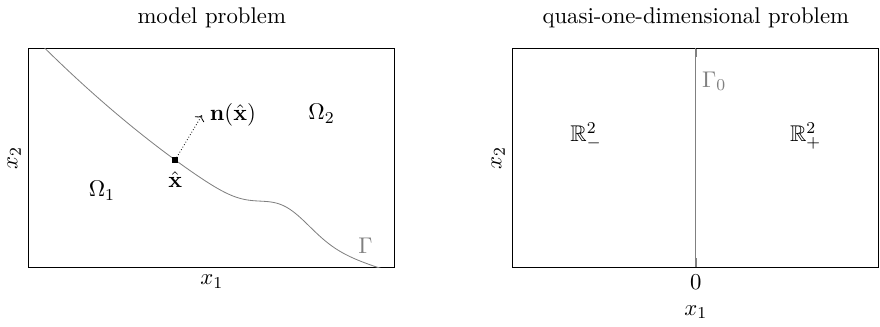}
  \caption{Domains of the model problem and the quasi-one-dimensional coupling problem in $\R^2$. While in the model problem two systems of conservation laws are coupled on a general interface the quasi-one-dimensional coupling originates from a localization of the problem at $\hat \bx \in \Gamma$.}\label{fig:2dcoupling}
\end{figure}

In this work we address this problem numerically. To handle the coupling we propose a local coordinate transform and projection of the fluxes for each interface point in the following. We refer to Figure~\ref{fig:2dcoupling} for a visual comparison of the domains corresponding to the model problem (left) and the projected quasi-one-dimensional local problem (right) in 2D.

\subsection{The global relaxation system}\label{sec:globalrelaxation}
In this section we introduce a Jin--Xin-type relaxation system, cf.~\cite{jin1995relaxschemsystem}, for the coupling problem from Section~\ref{sec:modelproblem}. To this end let $\bV^j(t, \bx)$ for $j \in \{1, \dots, d\}$ denote new auxiliary state variables as functions of time $t\in (0,\infty)$ and space $\bx \in \R^d$. These variables as well as the relaxation state $\bU$ maps to $\R^{m_1}$ if $\bx \in \Omega_1$ and to $\R^{m_2}$ if $\bx \in \Omega_2$. Both are governed by the system
\begin{equation}\label{eq:globrelaxationsystem}
  \left\{
    \begin{aligned}
      \frac{\partial \bU}{\partial t} + \sum_{j=1}^d\frac{\partial \bV^j}{\partial x_j} &= \bzero \qquad &&\text{for }(t, \bx) \in (0, \infty) \times \Omega \setminus \Gamma, \\
      \frac{\partial \bV^j}{\partial t} + (\bLambda_1^j)^2 \, \frac{\partial \bU}{\partial x_j} &= \frac{1}{\varepsilon} \left( \bF^j_1(\bU) - \bV^j \right) \qquad &&\text{for }j \in \{1, \dots, d\} \text{ and }(t, \bx) \in (0, \infty) \times \Omega_1, \\
       \frac{\partial \bV^j}{\partial t} + (\bLambda_2^j)^2 \, \frac{\partial \bU}{\partial x_j} &= \frac{1}{\varepsilon} \left( \bF^j_2(\bU) - \bV^j \right) \qquad &&\text{for }j \in \{1, \dots, d\} \text{ and }(t, \bx) \in (0, \infty) \times \Omega_2,
    \end{aligned}
  \right.
\end{equation}
in which $\varepsilon>0$ denotes the relaxation rate and $\bLambda_1^j \in \R^{m_1 \times m_1}$ as well as $\bLambda_2^j \in \R^{m_2 \times m_2}$ for $j \in \{1, \dots, d\}$ refer to diagonal matrices with positive diagonal-entries. Although the states $\bU$ and $\bV^j$ in \eqref{eq:globrelaxationsystem} depend on the relaxation rate $\varepsilon$ we neglect this dependency for brevity of notation. For the stability of our relaxation approach we require that the subcharacteristic conditions
\begin{equation}\label{eq:subcharacteristic}
  (\bLambda_i^j)^2 - (D\bF^j_i(\bU))^2 \geq \bzero \qquad \text{for }\bU \in \mathcal{D}_i,~ i \in \{1,2\}~ \text{and }j \in \{1, \dots, d\}
\end{equation}
hold in the sense of positive-semidefiniteness of the matrix on the left-hand side.

If we consider~\eqref{eq:globrelaxationsystem} restricted to either $\Omega_1$ or $\Omega_2$ and assume that~\eqref{eq:subcharacteristic} holds, an asymptotic analysis following the steps in~\cite{chen1994hyperconserlaws} shows that in the relaxation limit $\varepsilon \to 0$ the state $\bV^j$ coincides with $\bF^j_i(\bU)$ and the state $\bU$ satisfies the original multi-dimensional system of conservation laws
\[
   \frac{\partial \bU}{\partial t} + \sum_{j=1}^d \frac{\partial \bF_i^j(\bU)}{\partial x_j} = \bzero
\]
with $i=1$ in case of the negative half-space and $i=2$ in case of the positive half-space for any $j \in \{1, \dots, d\}$.

\begin{remark}
  For simplicity the matrices in~\eqref{eq:globrelaxationsystem} may be chosen independently of the direction, i.e.,
  \begin{equation}\label{eq:samelambdai}
    \bLambda_i^j = \bLambda_i \quad \text{for all }j \in\{1,\dots, d\} \text{ and } i\in\{1,2\}.
  \end{equation}
  Clearly, if matrices $\bLambda_i^j$ can be found that satisfy the subcharacteristic condition~\eqref{eq:subcharacteristic} taking the maximum over the direction $j$ with respect to the diagonal entries gives rise to a matrix $\bLambda_i$ that preserves this condition. Moreover, if we take $\lambda_i$ to be the maximal spectral norm of the directional Jacobians in system~\eqref{eq:system} on $\Omega_i$ then
  \begin{equation}\label{eq:relspeed}
    \bLambda_i = \lambda_i \, \bI \in \R^{m_i \times m_i} \quad \text{for }i\in\{1,2\}
  \end{equation}
  satisfies~\eqref{eq:subcharacteristic}, see \cite{kolbe2024numerschemcoupl}. In this situation we refer to $\lambda_1$ and $\lambda_2$ as \emph{relaxation speeds} of system~\eqref{eq:globrelaxationsystem}.
\end{remark}
To simplify the notation of the global relaxation system \eqref{eq:globrelaxationsystem} we denote by
\begin{equation}\label{eq:Q}
  \bQ=(\bU,~\bV^1,~\ldots,~\bV^d)^T \in \R^{(d+1)m_i}
\end{equation}
the combined state and define the vectors and matrices
\begin{equation}\label{eq:Aij}
  \mathcal{R}_i(\bQ) = \begin{pmatrix}
    \bzero \\
    \bF^1_i(\bU) - \bV^1 \\
    \vdots \\
    \bF^d_i(\bU) - \bV^d
    \end{pmatrix}, \quad
  \mathcal{A}_i^j = \begin{pmatrix}
    \bzero & \delta_{1j} \bI & \cdots & \delta_{dj} \bI \\
    \delta_{1j} (\bLambda_i^1)^2 & \bzero & \cdots &\bzero \\
    \vdots & \vdots & & \vdots \\
    \delta_{dj} (\bLambda_i^d)^2 & \bzero & \cdots & \bzero
    \end{pmatrix}
  \end{equation}
  for $i \in \{1,2\}$ and $j \in \{1, \dots, d\}$. Then system~\eqref{eq:globrelaxationsystem} recasts as
  \begin{equation}\label{eq:relaxationvector}
     \frac{\partial \bQ}{\partial t} + \sum_{j=1}^d \mathcal{A}_i^j \, \frac{\partial \bQ}{\partial x_j} =  \frac{1}{\varepsilon} \, \mathcal{R}_i(\bQ) \qquad \text{for }(t, \bx) \in (0, \infty) \times \Omega \setminus \Gamma,
   \end{equation}
   where the index $i=i(\bx)$ is adaptively taken such that $\bx \in \Omega_i$.

   To provide closure of the relaxation system~\eqref{eq:globrelaxationsystem} at the interface $\Gamma$ that is consistent with the coupling condition~\eqref{eq:couplingu} of the model problem~\eqref{eq:system} we consider a local transformation in the following subsection.

   \subsection{The localized relaxation system}\label{sec:locrelaxationsystem}
   In this section a normal projection of the global relaxation system~\eqref{eq:globrelaxationsystem} is considered that gives rise to a quasi-one-dimensional problem. The localized system is used to derive a suitable coupling condition for~\eqref{eq:globrelaxationsystem} and later to construct suitable numerical fluxes for our numerical method.

   At first, we fix a point $\hat \bx \in \Gamma$ and let $\bn = (n_1, \ldots, n_d)^T= \bn(\hat \bx)$ denote the unit normal of $\Gamma$ in $\hat \bx$. We further choose tangential vectors $\bt_1, \dots, \bt_{d-1}$ in $\hat \bx$ with respect to $\Gamma$ that together with $\bn$ form an orthonormal basis of $\R^d$. Thus $\bT_{\bn} = (\bn, \mathbf t_1, \dots, \mathbf t_{d-1})^T \in \R^{d \times d}$ is an orthogonal matrix enabling the coordinate transform  $\tilde \bx = \bT_{\bn} (\bx - \hat \bx)$. Thus, comparing the new coordinates to the original ones a shift in direction $\tilde x_1\eqqcolon \tilde x_\bn$ corresponds to a shift along $\bn$, whereas a shift in direction $\tilde x_j$ for $j\in\{2,\dots,d\}$ corresponds to a shift along a tangential direction to $\Gamma$ in $\hat \bx$.
   Next, we introduce the variable transforms
   \begin{equation}\label{eq:transform}
\begin{aligned}
  \tilde \bU(t, \tilde \bx) &= \bU(t, \bx(\tilde \bx)),\\
     \tilde \bV^\bn(t, \tilde \bx) \equiv \tilde \bV^1(t, \tilde \bx) &= \sum_{j=1}^d n_j\bV^j(t, \bx(\tilde \bx)) , \\
    \tilde \bV^{\bt_k}(t, \tilde \bx) \equiv \tilde \bV^{k+1}(t, \tilde \bx) &= \sum_{j=1}^d t_{k,j}\bV^j(t, \bx(\tilde \bx)) \quad \text{for }k \in \{1, \dots, d-1 \}
\end{aligned}
\end{equation}
with $t_{k,j}$ refering to component $j$ of $\mathbf t_k$ and $\bx(\tilde \bx) = \bT_{\bn}^T \tilde \bx + \hat \bx$. Combining the original and transformed relaxation variables in matrices, the transformations with respect to the auxiliary relaxation variables can be expressed as
\begin{equation}\label{eq:vtransform}
  (\tilde \bV^1, \dots, \tilde \bV^d)(t, \tilde \bx) = (\bV^1, \dots, \bV^d) (t, \bx(\tilde \bx)) \,\bT_{\bn}.
\end{equation}
We assume that there is no tangential flow along the interface, this implies that
\begin{equation}\label{eq:tangentialderivatives}
 \left. \frac{\partial \tilde \bU}{ \partial \tilde x_{j+1}} \right|_{(t, \bzero)} = \left. \frac{\partial \bV^{\bt_k}}{ \partial \tilde x_{j+1}} \right|_{(t, \bzero)} =\bzero \qquad \text{for }j,k \in \{1,\dots, d-1\} \quad \text{and a.e. } t\geq 0
\end{equation}
holds with respect to the left- and right-sided traces in $\tilde \bx=\bzero$, which is located on the coordinate transformed interface. %Hence, rewriting system \eqref{eq:system} in the new variable for $\tilde \bx = 0$ results in a local quasi-one-dimensional system.

The above transformation and localization at the interface point $\hat \bx$ motivates the quasi-one-dimensional relaxation system
   \begin{equation}\label{eq:relaxationprojected}
  \left\{
    \begin{aligned}
      \frac{\partial \tilde \bU}{\partial t} + \frac{\partial \tilde \bV^\bn}{\partial \tilde x_\bn} &= \bzero \qquad &&\text{for }(t, \bx) \in (0, \infty) \times \left( \R \setminus \Gamma_0 \right),\\
  \frac{\partial \tilde \bV^\bn}{\partial t} + (\bLambda_1^\bn)^2 \, \frac{\partial \tilde \bU}{\partial \tilde x_\bn} &= \frac{1}{\varepsilon} \left( \bF^\bn_1(\tilde \bU) - \tilde \bV^\bn \right) \qquad &&\text{for }(t, \bx) \in (0, \infty) \times \R^d_-,\\
      \frac{\partial \tilde \bV^\bn}{\partial t} + (\bLambda_2^\bn)^2 \, \frac{\partial \tilde \bU}{\partial \tilde x_\bn} &= \frac{1}{\varepsilon} \left( \bF^\bn_2(\tilde \bU) - \tilde \bV^\bn \right)  \qquad &&\text{for }(t, \bx) \in (0, \infty) \times \R^d_+
\end{aligned}
\right.
\end{equation}
which considers a quasi-one-dimensional coupling of the half-spaces $\R^d_- = (-\infty, 0) \times \R^{d-1}$ and $\R^d_+ = (0, \infty) \times \R^{d-1}$ at $\Gamma_0=\{0\} \times \R^{d-1}$, see Figure~\ref{fig:2dcoupling} (right). In analogy to the model \eqref{eq:system} its state variable $\bU$ is such that $ \bU(t, \bx) \in \mathcal{D}_1 $ if $\tilde x_\bn<0$ and such that $ \bU(t, \bx) \in \mathcal{D}_2$ if $\tilde x_\bn>0$. The remaining auxiliary variables connected to the tangential directions are governed by the ordinary differential equations
\begin{equation}\label{eq:relaxationprojected2}
  \left\{
    \begin{aligned}
      \frac{\partial \tilde \bV^{\bt_k} }{\partial t}  &= \frac{1}{\varepsilon} \left(\bF^{\bt_k}_1(\tilde \bU)  - \tilde \bV^{\bt_k} \right) \qquad &&\text{for }(t, \bx) \in (0, \infty) \times \R^d_-,\\
      \frac{\partial \tilde \bV^{\bt_k}}{\partial t}  &= \frac{1}{\varepsilon} \left(\bF^{\bt_k}_2(\tilde \bU) - \tilde \bV^{\bt_k} \right)  \qquad &&\text{for }(t, \bx) \in (0, \infty) \times \R^d_+
\end{aligned}
    \right.
  \end{equation}
  for $k \in \{1, \dots, d-1\}$, which follows from~\eqref{eq:tangentialderivatives}. The relaxation states $\tilde \bV^j$ for $j \in \{1, \dots, d\}$ map to $\R^{m_1}$ if $\tilde x_\bn < 0$ and to $\R^{m_2}$ if $\tilde x_\bn < 0$.

  While this system is defined on $\R^d$, it yields an approximation of system~\eqref{eq:globrelaxationsystem} in a neighborhood around $\hat \bx$ if it is constrained to a neighborhood around $\bx = \bzero$ after inverting the above change of variables. The flux functions employed in~\eqref{eq:relaxationprojected} are defined by
\begin{equation}\label{eq:normalflux}
\bF_i^{\bn} (\bU) \coloneqq \sum_{j=1}^d  \bF_i^j(\bU) \, n_j \qquad \text{for }i \in \{1,2\}
\end{equation}
and referred to as \emph{normal fluxes} in the following. Analogously the \emph{tangential fluxes} appearing in~\eqref{eq:relaxationprojected2} are given by

\[
  \bF^{\bt_k}_i = \sum_{j=1}^d \bF^j_i(\tilde \bU) \, t_{k,j} \qquad \text{for }i \in \{1,2\} \text{ and } k \in \{1, \dots, d-1\}.
\]
Similarly, the matrices $\bLambda_i^\bn$ for $i \in \{1,2\}$ are the diagonal matrices with positive entries satisfying
\[
  (\bLambda_i^\bn)^2 = \sum_{j=1}^d n_j \, (\bLambda_i^j)^2.
\]
Combining the original state and the normal relaxation state in the variable $\tilde \bQ^\bn = (\tilde \bU, \tilde \bV^\bn)$ we state the coupling condition for the localized system~\eqref{eq:relaxationprojected} as
\begin{equation}\label{eq:locrelcoupling}
  \Psi_Q^\bn (\tilde \bQ^\bn(t, \tilde \bx^-) , \tilde \bQ^\bn(t, \tilde \bx^+)) = \bzero \qquad \text{for a.e. } t \geq 0 \quad \text{and } \tilde \bx \in \Gamma_0,
\end{equation}
where the vector-valued map $\Psi_Q^\bn : \mathcal{D}_1 \times \R^{m_1} \times \mathcal{D}_2 \times \R^{m_2} \to \R^{\tilde \ell}$ for a problem-specific $\tilde \ell \in \N$ encodes the relevant coupling rules on both state variables at the interface. 
Note that unlike~\eqref{eq:couplingu} the normal in \eqref{eq:relcoupling} is independent of $\bx\in \Gamma_0$ but fixed with the choice of $\hat \bx$.

The quasi-one-dimensional coupled system given by \eqref{eq:relaxationprojected}, \eqref{eq:relaxationprojected2} and \eqref{eq:locrelcoupling} depends on the point $\hat \bx \in \Gamma$ within the original coupling problem from Section~\ref{sec:modelproblem} in the sense that the normal of the interface in $\hat \bx$ is used within the flux \eqref{eq:normalflux} and initial data will be adopted from the state $\bU$ in a neighborhood of $\hat \bx$.

\begin{remark}
  We have derived the localized relaxation system~\eqref{eq:relaxationprojected} by first relaxing the multi-dimensional system, second conducting a change of variables and third projecting into the normal direction. The order of these operations can be varied, e.g.\ relaxing the system after the change of variables or even in the last step after the change of variables and projection in normal direction; system~\eqref{eq:relaxationprojected} remains invariant of the order of those operations in the derivation. 
\end{remark}

\subsection{The relaxed coupling condition}\label{sec:relaxedcoupling}
Making use of the localization in Section~\ref{sec:locrelaxationsystem} we can close the global relaxation system~\eqref{eq:globrelaxationsystem} at the interface. To this end for any $\bx \in \Gamma$ with interface unit normal $\bn=\bn(\bx)$ let $\bV^\bn = \sum_{j=1}^d n_j \bV^j$ refer to the normal relaxation state in original coordinates and $\bQ^\bn = (\bU, \bV^\bn)$ to the combined normal relaxation state. Employing the map $\Psi_Q^\bn$ from~\eqref{eq:locrelcoupling} we impose the coupling condition
\begin{equation}\label{eq:relcoupling}
  \Psi_Q^\bn (\bQ^\bn(t, \bx^-), \bQ^\bn(t, \bx^+)) = \bzero \qquad \text{for a.e. } t \geq 0 \quad \text{and } \bx \in \Gamma,
\end{equation}
where through the normal vector $\bn$ the input states depend on the interface point $\bx$ and an analogous limit notation as in~\eqref{eq:couplingu} is used. We refer to~\eqref{eq:relcoupling} as \emph{relaxed coupling condition} in the following.

To consistently approximate the coupled system by the relaxation approach it is necessary that the relaxed coupling condition \eqref{eq:relcoupling} relates to the original condition \eqref{eq:couplingu}. Therefore, we adopt the result from our one-dimensional analysis in \cite{herty2023centrschemtwo} in this section.
Assuming the map $\Psi_Q^\bn$ to be smooth for all $\bn \in \R^d$ satisfying $\| \bn\|=1$ and taking the relaxation limit of a Chapman--Enskog expansion, see \cite{chapman1990mathemtheornonuniforgases}, motivates the following notion of consistency.

\begin{definition}\label{def:consistent}
For given $\bx \in \Gamma$ and $\bn=\bn( \bx)$ let $\bQ^\bn_i(\bU)$ refer to the combined states of~\eqref{eq:globrelaxationsystem} in the relaxation limit on $\Omega_i$ for $i \in \{1,2\}$, i.e. $\bQ^\bn_i(\bU) = (\bU, \bF^\bn_i(\bU))$. We say that the coupled relaxation system~\eqref{eq:globrelaxationsystem} is \emph{consistent} with the coupled problem~\eqref{eq:system} in $\bx \in \Gamma$ if the equivalence
\begin{equation}\label{eq:consistent}
  \Psi_U^\bn(\bU^- , \bU^+) = \bzero \quad \Leftrightarrow \quad \Psi_Q^\bn(\bQ^\bn_1 ( \bU^-), \bQ^\bn_2( \bU^+ )) = \bzero
\end{equation}
applies for all $\bU^- \in \mathcal{D}_1$ and $\bU^+ \in \mathcal{D}_2$ with respect to the coupling conditions~\eqref{eq:couplingu} and~\eqref{eq:relcoupling}. Moreover, if~\eqref{eq:consistent} holds for all $\bn \in \R^d$ with  $\| \bn\|=1$ we call both families of coupling functions, $(\Psi_U^\bn)_\bn$ and $(\Psi_Q^\bn)_\bn$ consistent.
\end{definition}

We use Definition~\ref{def:consistent} to construct problem-suited Riemann solvers for the numerical discretization in the following section.

\section{The Riemann problem at the interface}\label{sec:riemann}
The aim of this section is to construct a \emph{Riemann solver} for the quasi-one-dimensional coupled relaxation system introduced in Section~\ref{sec:locrelaxationsystem}. This Riemann solver will be employed in the multi-dimensional relaxation system \eqref{eq:globrelaxationsystem} and finally serve as a key component of the numerical scheme for the original coupling problem~\eqref{eq:system}. Here we discuss the general construction and some of its properties; details will depend on the problem at hand.

Throughout this section we fix an interface point $\hat \bx$ and the corresponding unit normal $\bn=\bn(\hat \bx)$. Our discussion focusses on the coupling problem given by \eqref{eq:relaxationprojected} and \eqref{eq:locrelcoupling}, where for simplicity we neglect the tilde notation and set $\bV \coloneqq \bV^\bn$ as well as by abuse of notation $\bQ \coloneqq \bQ^\bn$. Furthermore, we adopt the following terminology in analogy to~\cite{garavello2006traffflownetwor,herty2023centr}.

\begin{definition}\label{def:rp}
  The \emph{Riemann problem} for the quasi-one-dimensional coupled relaxation system is given by \eqref{eq:relaxationprojected} with coupling condition \eqref{eq:locrelcoupling} and constant initial data on both half planes, i.e.
  \[
    (\bU, \bV) (0, \bx) =
    \begin{cases*}
      (\bU^-, \bV^-) & if $\bx \in \R^d_-$, \\
      (\bU^+, \bV^+) & if $\bx \in \R^d_+$
    \end{cases*}
  \]
  for $(\bU^-, \bV^-) \in \mathcal{D}_1 \times \R^{m_1}$ and $(\bU^+, \bV^+) \in \mathcal{D}_2 \times \R^{m_2}$.
\end{definition}

\begin{figure}
  \centering
  \includegraphics{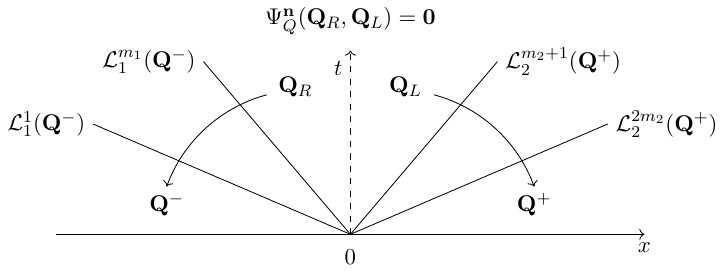}
 \caption{The coupled half-Riemann problem for the localized relaxation system in the $x$-$t$-plane. The left trace data $\bQ^{-}$ are connected to the coupling data $\bQ_R$ by Lax-curves of negative speeds, i.e., $\mathcal{L}_{1}^1\dots\mathcal{L}_{1}^{m_1}$ and the outgoing trace data $\bQ^{+}$ are connected to the coupling data $\bQ_L$ by Lax-curves of positive speeds, i.e., $\mathcal{L}_{2}^{m_2+1}\dots \mathcal{L}_2^{2m_1}$. Coupling data at the interface are related by the coupling condition $\Psi_Q^\bn$.}\label{fig:interfacewaves}
\end{figure}

\begin{definition}\label{def:rs} A \emph{Riemann solver} (RS) for the Riemann problem in Def.~\ref{def:rp} is a function $\rs$ mapping from $\mathcal{D}_1 \times \R^{m_1} \times \mathcal{D}_2 \times \R^{m_2}$ to itself,
  \[
    \rs: (\bU^-, \bV^-, \bU^+, \bV^+) \mapsto (\bU_R, \bV_R, \bU_L, \bV_L),
  \]
  such that the following properties apply:
  \begin{enumerate}
  \item The state $\bQ_R = (\bU_R, \bV_R)$ is a suitable right boundary datum of the Riemann problem at the interface in the sense that it connects to $\bQ^- = (\bU^-, \bV^-)$ on $\R^d_-$ by Lax curves that correspond to characteristics with negative speed. Analogously, $\bQ_L= (\bU_L, \bV_L)$ is a suitable left boundary datum at the interface and thus connects to $\bQ^+ = (\bU^+, \bV^+)$ on $\R^d_+$ by Lax curves that correspond to characteristics with positive speed.
  \item The \emph{coupling data} $\bQ_R$ and $\bQ_L$ satisfy the coupling condition, i.e.
    \begin{equation}\label{eq:statecoupling}
      \Psi_Q^\bn(\bQ_R, \bQ_L) = \bzero.
    \end{equation}
  \item The function $\rs$ is idempotent, i.e.\ if we take $(\bU_R, \bV_R, \bU_L, \bV_L)$ from its image, it holds
    \[
      \rs\left( \bU_R, \bV_R, \bU_L, \bV_L \right) = (\bU_R, \bV_R, \bU_L, \bV_L).
      \]
  \end{enumerate}
\end{definition}

Figure~\ref{fig:interfacewaves} illustrates the relation between the trace states $\bQ^-$ and $\bQ^+$ to the coupling data $\bQ_R$ and $\bQ_L$ given a RS satisfying Definition~\ref{def:rs}.

\subsection{The half-Riemann problem}\label{sec:halfRP}
We first address the half-Riemann problem that is a component of the first property in Def.~\ref{def:rs}. More precisely, it inquires after the set of valid boundary data at $\Gamma_0$ if the relaxation system \eqref{eq:relaxationprojected} is imposed either to the left or the right half-space.

Therefore, we conduct an eigenvalue analysis in analogy to the procedure in~\cite{herty2023centrschemtwo}. Clearly,~\eqref{eq:relaxationprojected} is a linear hyperbolic system with block form system matrix, which can be diagonalized
\begin{equation}\label{eq:diagonalized}
  \begin{pmatrix}
    \bzero & \bI \\
    (\bLambda_i^n)^2 & \bzero
  \end{pmatrix}
  = \mathbf R_i \mathbf D_i \mathbf R_i^{-1}, \qquad
  \mathbf R_i =
  \begin{pmatrix}
    -(\bLambda_i^\bn)^{-1} & (\bLambda_i^\bn)^{-1} \\
    \bI & \bI
  \end{pmatrix}, \quad
\mathbf D_i=
  \begin{pmatrix}
    -\bLambda_i^\bn & \bzero \\
    \bzero & \bLambda_i^\bn
  \end{pmatrix}.
\end{equation}
The columns of the matrix $\mathbf R_i$ give rise to the \emph{Lax curves} of the relaxation system, i.e.\ the locus comprising the states that given a starting point can be reached by a simple wave. In more details, the Lax curve corresponding to the $q$-wave for $q \in \{ 1, \dots, 2 m_i\}$ and the initial state $\bQ_0 = (\bU_0, \bV_0)$ reads
\[
  \mathcal{L}_i^q(\bQ_0; \sigma_q) =
  \begin{cases*}
    (\bU_0 - \sigma_q \mathbf e_q, \bV_0 + \sigma_q \lambda_q \mathbf e_q) & if $1 \leq q\leq m_i$ \\
    (\bU_0 + \sigma_{q} \mathbf e_{q-m_i}, \bV_0 + \sigma_q \lambda_{q-m_i} \mathbf e_{q-m_i}) & if $m_i+1 \leq q \leq 2 m_i$
  \end{cases*},
\]
where for $1\leq j \leq m_i$ $\mathbf e_j$ denotes the $j$-th unit vector in $\R^{m_i}$ and $\lambda_j$ the $j$-th diagonal entry of $\bLambda_i^\bn$ and $\sigma_q \in \R$ parameterizes the curve for $1\leq q \leq 2m_i$.

Following from the signs of the eigenvalues in \eqref{eq:diagonalized} $\bQ_R$ satisfies the first property of Def.~\ref{def:rs} if it is included in the span of the curves $\mathcal{L}_1^1(\bQ^-),\dots, \mathcal{L}_1^{m_1}(\bQ^-)$. Similarly, $\bQ_L$ is a suitable left boundary datum in the sense of Def.~\ref{def:rs} if it is included in the span of the curves $\mathcal{L}_2^{m_2 + 1}(\bQ^+),\dots, \mathcal{L}_1^{2 m_2}(\bQ^+)$. In other words, the above property holds iff there exist two vectors $\bSigma_1 \in \R^{m_1}$ and $\bSigma_2 \in \R^{m_2}$ such that
\begin{equation}\label{eq:parameterized}
  \bQ_R = (\bU^- - \bSigma_1, \bV^- + \bLambda_1^\bn \bSigma_1) \quad \text{and} \quad  \bQ_L = (\bU^+ + \bSigma_2, \bV^+ + \bLambda_2^\bn \bSigma_2).
\end{equation}

\subsection{Consistent Riemann solvers}
Given the parameterization \eqref{eq:parameterized} and the coupling function $\Psi_\bQ^\bn$ the coupling condition for the coupling data within the second condition of Def.~\ref{def:rs} implies a generally nonlinear system in the variables $\bSigma_1$ and $\bSigma_2$. Depending on the coupling function it is not clear, whether a solution to this system exists; in fact, unique solvability for any $\bQ^-$ and $\bQ^+$ is required to obtain a well-posed RS. Some results concerning its well-posedness in case of nonlinear coupling functions have been provided in \cite{herty2023centrschemtwo} along with an explicit representation of the coupling data in case of affine linear coupling functions.

Since we aim for a numerical solution of the unrelaxed model \eqref{eq:system} with coupling condition~\eqref{eq:couplingu} we are only interested in coupling functions $\Psi_\bQ^\bn$ that have been derived in consistency with the original coupling function $\Psi_\bU^\bn$ within condition \eqref{eq:couplingu}. Thus, we propose the following three-step procedure for the general construction of an RS:
\begin{algorithm}[Construction of the RS]\label{algo:rsconstruction}
~\\[-4mm]
\begin{enumerate}
\item To derive a relaxation coupling condition consistent with~\eqref{eq:couplingu} define
  \begin{equation}\label{eq:psiqconstruction}
    \Psi_\bQ^\bn((\bU^-, \bV^-), (\bU^+, \bV^+)) = (\Psi_\bU^\bn(\bU^-, \bU^+),  \Psi_R^\bn((\bU^-, \bV^-), (\bU^+, \bV^+))),
  \end{equation}
  i.e.\ copy the original coupling condition in the first components of the modified coupling condition and introduce the map $\Psi_R^\bn: \mathcal{D}_1 \times \R^{m_1} \times \mathcal{D}_2 \times \R^{m_2} \to \R^{\tilde \ell- \ell}$ accounting for the remaining conditions.
\item Populate the remainder $ \Psi_R^\bn((\bU^-, \bV^-), (\bU^+, \bV^+))$ with $\tilde \ell - \ell$ new conditions relating the states $\bV^-$ and $\bV^+$ by using the original condition~\eqref{eq:couplingu} and the limit properties $\bV^- = \bF^n_1(\bU^-)$ and $\bV^+ = \bF^n_2(\bU^+)$. To make consistent statements with the original coupling conditions it may be necessary to also include the original states $\bU^-$ and $\bU^+$.
\item In the system following from \eqref{eq:statecoupling} include \eqref{eq:parameterized} by substituting
  \begin{equation}\label{eq:laxrelation}
    \bV_R - \bV^- = \bLambda_1(\bU^- - \bU_R)  \quad \text{and} \quad  \bV_L - \bV^+ = \bLambda_2(\bU_L - \bU^+)
  \end{equation}
  and hence eliminating the variables $\bV_R$ and $\bV_L$. The result is a generally nonlinear system in $m_1 + m_2$ scalar unknowns and its solution governs the output of the RS.
\end{enumerate}
\end{algorithm}
\begin{proposition} Suppose that $\Psi_Q^\bn$ is such that the system obtained in step 2 of Algorithm~\ref{algo:rsconstruction} has a unique solution. Then the algorithm defines a RS.
\end{proposition}
\begin{proof} Properties 1 and 2 in Def~\ref{def:rs} clearly follow from the construction in the algorithm and the well-posedness from the unique solvability of the governing system. To verify property~3 let $(\bU^-, \bV^-, \bU^+, \bV^+)$ denote the initial input data into the RS, $(\bU_R^1, \bV_R^1, \bU_L^1, \bV_L^1)$ the corresponding output and
  \[
    (\bU_R^2, \bV_R^2, \bU_L^2, \bV_L^2) = \rs(\bU_R^1, \bV_R^1, \bU_L^1, \bV_L^1).
  \]
  Since we have $\Psi_\bU^\bn(\bU_R^1, \bU_L^1)=\bzero$ due to the construction of the involved states, choosing $\bU_R^2 = \bU_R^1$ and $\bU_L^2 = \bU_L^1$ solves the first $\ell$ equations within the system $\Psi_\bQ^\bn(\bQ_R^2, \bQ_L^2)=\bzero$ because of~\eqref{eq:psiqconstruction}. This choice, moreover, implies $\bV_R^2 = \bV_R^1$ and $\bV_L^2 = \bV_L^1$ due to~\eqref{eq:laxrelation} and thus the substitution in step 3 leads to the same system as the one solved by the states $\bU_R^1$ and $\bU_L^1$ when they were computed. Considering the unique solvability of the governing system implies property~3.
\end{proof}

\section{Discontinuous Galerkin schemes}\label{sec:DG}
In this section we derive fully-discrete schemes for the coupled problem~\eqref{eq:system}. Those are based on a discontinuous Galerkin (DG) approximation in space and Runge--Kutta methods in time.
For their derivation we first introduce discrizations of the relaxation system~\eqref{eq:globrelaxationsystem} in space and time in Sections~\ref{sec:spacedisc} and \ref{sec:timedisc}. The final schemes for the original problem are then obtained taking the discrete relaxation limit in Section~\ref{sec:relaxationlimit}.

\subsection{The space discretization}\label{sec:spacedisc}
As a starting point, we discretize the global relaxation system~\eqref{eq:globrelaxationsystem} on the bounded computational domain $\Omega^c \subset \R^d$. This requires boundary conditions that need to be specified with respect to the application, see Section~\ref{sec:example}. We assume that $\Omega^c \cap \Gamma \neq \emptyset$ and define the computational subdomains $\Omega_1^c = \Omega_1 \cap \Omega ^c$ and $\Omega_2^c = \Omega_2 \cap \Omega ^c$ such that $\overline{\Omega_1^c} \cap \overline{\Omega_2^c} = \Gamma.$

We partition both subdomains by a finite number of cells $C_\mu$ such that the (disjoint) unions
\[
  \bar \Omega_1^c = \overline{\bigsqcup_{\mu \in \mathcal{I}_1} C_\mu} \quad \text{and} \quad \bar \Omega_2^c = \overline{\bigsqcup_{\mu \in \mathcal{I}_2} C_\mu}
\]
hold, where $\mathcal{I}_1$ and $\mathcal{I}_2$ are two disjoint index sets, whose union we denote by $\mathcal{I}= \mathcal{I}_1 \sqcup \mathcal{I}_2$. This partition implies that some of the cells have parts of their boundaries located on $\Gamma$ and $\mathring C_\mu \cap \Gamma = \emptyset$ for all $\mu \in \mathcal{I}$. We introduce the DG spaces
\begin{align*}
  \mathcal{S}_{i} &\coloneqq \{ f \in L^2(\Omega_i^c): ~ f \rvert_{C_\mu} \in \Pi_{p-1}(C_\mu), ~ \mu \in \mathcal{I}_i \} \qquad \text{for }i \in \{1,2\}, \\
  \mathcal{S}&\coloneqq \{ f \in L^2(\Omega^c):~f\rvert_{\Omega_1^C} \in \mathcal{S}_1 \text{ and } f\rvert_{\Omega_2^C} \in \mathcal{S}_2\}
\end{align*}
denoting by $\Pi_{p}(C_\mu)$ the space of polynomial functions on $C_\mu$ that are of degree smaller than $p \in \N$, i.e.,
\[
  \Pi_{p-1}(C) = \operatorname{span} \left\{ \mathbf x^{\mathbf j} \rvert_C:~ \mathbf j \in \N_0^d \text{ with } \|\mathbf j \|_\infty \leq p-1 \right\}.
\]
Here we use the notation $\mathbf x^{\mathbf j}=\Pi_{i=1}^d x_i^{j_i}$ to refer to monomials in $\R^d$. In the following the space of suitable multi-indices is denoted by $\mathcal{P} = \{ \mathbf j \in \N_0^d :~ \|\mathbf j \|_\infty \leq p-1\}$. Next, let $\Phi = \{ \mathbf \phi_{\mu, \mathbf j}: ~ \mu \in \mathcal{I}, ~ \mathbf j \in \mathcal{P} \}$ denote a basis of $\mathcal{S}$ with basis elements satisfying $\operatorname{supp}(\mathbf \phi_{\mu, \mathbf j}) = C_\mu$ that is orthogonal with respect to the standard inner product in $L^2(\Omega^c)$, i.e.\
\[
  \langle \mathbf \phi_{\mu, \mathbf j} , \mathbf \phi_{\mu^\prime, \mathbf j^\prime} \rangle_{L^2(\Omega^c)} = \delta_{\mu, \mu^\prime}~\delta_{\mathbf j, \mathbf j^\prime}.
\]
A semi-discretization of the combined state~\eqref{eq:Q} in system~\eqref{eq:relaxationvector} over $\Omega^c$ is then given by
\begin{equation}\label{eq:DGS}
\bQ_h(t, \bx) = \sum_{\mu \in \mathcal{I}} \sum_{\mathbf j \in \mathcal{P}} \bQ_{\mu, \bj}(t) \phi_{\mu, \bj}(\bx),
\end{equation}
where $\bQ_{\mu, \bj}: (0, T) \to \R^{(d+1)m_i}$ for $i$ such that $\mu \in \mathcal{I}_i$ and $T>0$ is the fixed final time. In analogy to \eqref{eq:Q} we introduce the partial coefficients $\bU_{\mu, \mathbf j}$, $\bV^j_{\mu, \mathbf j}:(0, T) \to \R^{m_i}$ as well as the partial semi-discrete states $\bU_h(t, \mathbf x)$ and $\bV^j_h(t, \mathbf x)$ for $j\in\{1,\dots,d\}$.

The DG scheme is obtained substituting~\eqref{eq:DGS} in the uniform relaxation system~\eqref{eq:globrelaxationsystem} testing with $w_h \in \mathcal{S}$ and making use of integration by parts: Find $\bQ_h(\cdot, t)$ for $t \in [0, T]$ such that for all $w_h \in \mathcal{S}$ and $\mu \in \mathcal{I}$ it holds
\begin{equation}\label{eq:DGweak}
  \int_{C_\mu} \frac{\partial \bQ_h}{\partial t} w_h \, dx + \int_{\partial C_\mu} \hat{\mathbf f}_\mu(\bQ_h^-, \bQ_h^+; \bn) w_h \, dS - \int_{C_\mu} \sum_{j=1}^d \mathcal{A}_i^j \bQ_h \frac{\partial w_h}{\partial x_j} \, dx = \frac{1}{\varepsilon} \int_{C_\mu} \mathcal{R}_i(\bQ_h) w_h \, dx
\end{equation}
for a numerical flux function $\hat{\mathbf f}$ and $i$ such that $\mu \in \mathcal{I}_i$. The arguments $\bQ_h^-$ and $\bQ_h^+$ of the numerical flux function refer to the inner and outer value of $\bQ_h$ with respect to the cell $C_\mu$ at any interface point and $\bn$ refers to the outer unit normal.

We choose Godunov/Upwind fluxes approximating the linear balance law~\eqref{eq:relaxationvector}. Let at first $\mu\in \mathcal{I}_i$ and the point $\bx \in \partial C_\mu$, at which we evaluate the numerical flux function, not be on the interface, i.e. $\bx \notin \Gamma$ and such that it is located at the boundary of exactly two cells\footnote{The quadrature formulas approximating the boundary integral will be chosen such that the quadrature points always satisfy the latter condition.}. In this case the numerical flux function can be written in flux-vector splitting form
\begin{equation}\label{eq:fluxdirsplitting}
\hat{\mathbf f}_\mu(\bQ_h^-, \bQ_h^+; \bn) = \mathcal{A}_i^{\bn, +} \, \bQ_h^- + \mathcal{A}_i^{\bn, -} \, \bQ_h^+.
\end{equation}
Here the matrices are based on the definition and the diagonalization
\begin{equation}\label{eq:Ain}
  \mathcal{A}_i^\bn \coloneqq \sum_{j=1}^d  n_j\, \mathcal{A}_i^j = \mathcal{R}_i^\bn \, \mathcal{D}_i^\bn \, (\mathcal{R}_i^\bn)^{-1}
\end{equation}
with $\mathcal{A}_i^j$ given in~\eqref{eq:Aij} and defined as $\mathcal{A}_i^{\bn,\pm} = \mathcal{R}_i^\bn \, \mathcal{D}_i^{\bn, \pm} \,(\mathcal{R}_i^\bn)^{-1}$, where $\mathcal{D}_i^{\bn, +}$ refers to $\mathcal{D}_i^\bn$ with all negative diagonal entries replaced by zero and analogously, $\mathcal{D}_i^{\bn, -}$ to $\mathcal{D}_i^\bn$ with all positive diagonal entries replaced by zero. We refer to Appendix~\ref{appx:diagoonalization} for the detailed matrices. Let $C_{\tilde \mu}$ refer to the adjacent cell, such that $\bx \in \partial C_{\tilde \mu}$ and $\mu \neq \tilde \mu$. If $\bn$ points towards $C_{\tilde \mu}$ the definition of the numerical flux implies that $\hat{\mathbf f}_\mu(\bQ_h^-, \bQ_h^+; \bn) = -\hat{\mathbf f}_{\tilde \mu}(\bQ_h^+, \bQ_h^-; -\bn)$.

\begin{remark}\label{rem:relaxationfluxformulas}
  Splitting the numerical flux with respect to the original state $\bU$ and the relaxation variables $\bV^1, \dots, \bV^d$ as
  \[
    \hat{\mathbf f}_\mu(\bQ_h^-, \bQ_h^+; \bn) = (\hat{\mathbf f}_\mu^\bU \,  \hat{\mathbf f}_\mu^{\bV^1} \ldots \hat{\mathbf f}_\mu^{\bV^d})^T (\bQ_h^-, \bQ_h^+; \bn)
  \]
  and considering again $\bx \notin \Gamma$ we observe that
  \begin{equation}\label{eq:relaxuflux}
    \hat{\mathbf f}_\mu^\bU(\bQ_h^-, \bQ_h^+; \bn) = \frac 12 \bSigma_i \left( \bU_h^- - \bU_h^+\right) + \frac 12  \left(\bV_h^{\bn,-} + \bV_h^{\bn,+}\right)
  \end{equation}
  and
  \begin{equation}\label{eq:relaxvflux}
    \hat{\mathbf f}_\mu^{\bV^j}(\bQ_h^-, \bQ_h^+; \bn) = \frac 12 \bSigma_i^{-1} (\bLambda_i^j)^2 \left( \bV_h^{\bn,-} - \bV_h^{\bn,+}\right) n_j + \frac{1}{2}  (\bLambda_i^j)^2  \left(\bU_h^{-} + \bU_h^{+}\right) n_j
  \end{equation}
  for $j \in \{1,\dots, n\}$. We have employed here the notations
  \begin{equation}\label{eq:Vhn}
    \bV_h^{\bn,\pm} = \sum_{j=1}^d n_j \bV_h^{j,\pm}\quad \text{and} \quad \bSigma_i = \sqrt{\sum_{j=1}^d n_j^2 (\bLambda_i^j)^2}
  \end{equation} and note that if~\eqref{eq:samelambdai} holds the fluxes~\eqref{eq:relaxuflux} and~\eqref{eq:relaxvflux} further simplify as $\bSigma_i=\bLambda_i^j$ for $j \in \{1,\dots,d\}$.
\end{remark}

To define the numerical flux perpendicular to the interface we consider two adjacent cells $C_\mu$ and $C_{\tilde \mu}$ that share a part of their common boundary with the interface. We fix the point $\hat \bx \in \Gamma \cap \partial C_\mu \cap \partial C_{\tilde \mu}$ and refer by $\bn$ to the unit normal in this point with respect to $\Gamma$. We assume that $\mu \in \mathcal{I}_1$, $\tilde \mu \in \mathcal{I}_2$ and $\bn$ points into $C_{\tilde \mu}$. Next, for a fixed time instance $t$ let $\bQ_h^- \in \R^{d(m_1 +1)}$ refer to the trace of the numerical solution from $\Omega_1^c$ in $\hat \bx$ and $\bQ_h^+\in \R^{d(m_2 +1)}$ to the trace of the numerical solution from $\Omega_2^c$ in $\hat \bx$, i.e.,
\begin{equation}\label{eq:interfacetraces}
  \bQ_h^- = \sum_{\bj \in \mathcal{P}} \bQ_{\mu, \bj}(t) \phi_{\mu, \bj}(\hat \bx) \quad \text{and} \quad \bQ_h^+ = \sum_{\bj \in \mathcal{P}} \bQ_{\tilde \mu, \bj}(t) \phi_{\tilde \mu, \bj}(\hat \bx).
\end{equation}
To compute the numerical flux in $\hat \bx$ we locally consider the projection introduced in Section~\ref{sec:locrelaxationsystem} with respect to this interface point and the normal $\bn$. Applying the variable transform~\eqref{eq:transform} we obtain the states $\tilde \bU^\mp$ and $\tilde \bV^{k,\mp}$ for $k \in \{1, \dots, d\}$ from $\bQ_h^\mp$. In the next step, suitable coupling data with respect to this trace data and the projected relaxation system given by~\eqref{eq:relaxationprojected} and \eqref{eq:relaxationprojected2} is computed. As no flow occurs in the directions $\mathbf t_1, \dots, \mathbf t_{d-1}$ and the corresponding relaxation variables are governed by system~\eqref{eq:relaxationprojected2} we set
\begin{equation}
\tilde \bV^{\bt_k}_R = \tilde \bV^{\bt_k,-} \quad \text{and} \quad \tilde \bV^{\bt_k}_L = \tilde \bV^{\bt_k,+} \quad \text{for }k \in \{1, \dots d-1\}.
\end{equation}
Coupling data regarding the relaxed state $\tilde \bU$ and the normal relaxation variable $\tilde \bV^\bn$ is obtained evaluating the RS constructed in Section~\ref{sec:riemann} and taking
\begin{equation}
(\tilde \bU_R, \tilde \bV_R^\bn, \tilde \bU_L, \tilde \bV_L^\bn) = \rs(\tilde \bU^-, \tilde \bV^{\bn,-}, \tilde \bU^+, \tilde \bV^{\bn,+}).
\end{equation}
Having derived the coupling data for the projected relaxation system we revert the variable transform using~\eqref{eq:vtransform}, which gives rise to the coupling states $\bU_{h,R}$, $\bU_{h,L}$, $\bV^j_{h,R}$, $\bV^j_{h,L}$ for $j \in \{1, \dots, d\}$ and therefore $\bQ_{h,R}$, $\bQ_{h,L}$ in original variables. Those are eventually used to evaluate the numerical flux functions as
\begin{equation}
  \hat{\mathbf f}_\mu(\bQ_h^-, \bQ_h^+; \bn) \coloneqq \hat{\mathbf f}_\mu(\bQ_h^-, \bQ_{h,R}; \bn) \quad \text{and} \quad %\hat{\mathbf f}_{\tilde \mu}(\bQ_h^-, \bQ_h^+, \bn) \coloneqq \hat{\mathbf f}_{\tilde \mu}(\bQ_{h,L}, \bQ_{h}^+, \bn),
  \hat{\mathbf f}_{\tilde \mu}(\bQ_h^+, \bQ_h^-; -\bn) \coloneqq \hat{\mathbf f}_{\tilde \mu}(\bQ_{h}^+, \bQ_{h,L}; -\bn),
\end{equation}
where common flux-vector splitting~\eqref{eq:fluxdirsplitting} is employed to compute the right-hand side expressions, see also Remark~\ref{rem:relaxationfluxformulas} for an explicit form of the numerical fluxes.

A compact semi-discrete scheme is obtained, choosing the test function $w_h = \phi_{\mu, \bj}$ in~\eqref{eq:DGweak} resulting in
\begin{align}\label{eq:DGsemidiscrete}
  \frac{\partial \bQ_{\mu, \bj}}{\partial t}  &= \int_{C_\mu} \sum_{j=1}^d \mathcal{A}_i^j \bQ_h \frac{\partial \phi_{\mu, \bj}}{\partial x_j} \, dx  - \int_{\partial C_\mu} \hat{\mathbf f}_\mu(\bQ_h^-, \bQ_h^+, \bn) \phi_{\mu, \bj} \, dS + \frac{1}{\varepsilon} \int_{C_\mu} \mathcal{R}_i(\bQ_h) \phi_{\mu, \bj} \, dx \notag \\
  &\coloneqq G_{\mu, \bj}(\bQ_h) - B_{\mu, \bj}(\bQ_h) + \frac{1}{\varepsilon} R_{\mu, \bj}(\bQ_h)
\end{align}
for all $\mu \in \mathcal{I}$ and $\bj \in \mathcal{P}$. The latter short-hand notation makes use of volume and boundary integrals depending on $\bj \in \mathcal{P}$ and $\mu \in \mathcal{I}$ and $\bQ_h$ with $\mu \in \mathcal{I}_i$ determining the corresponding computational subdomain.

\subsection{The time discretization}\label{sec:timedisc}
To derive a fully discrete scheme for the coupled relaxation system~\eqref{eq:globrelaxationsystem} a time discretization of~\eqref{eq:DGsemidiscrete} is required. We consider a partition of the time interval given by $t^n = \sum_{k=0}^{n-1} \Delta t^k$ with positive increments $\Delta t^n$. Additionally we employ the notations $\bQ_h^n$ and $\bQ_{\mu, \bj}^n$ to refer to our DG approximation and the corresponding coefficients at time instance $t^n$. We use analogue notations to refer to the components of the DG approximation.

In the finite volume setting an unsplit scheme for the relaxation scheme was proposed in~\cite{jin2012asympap}. An adaptation to our DG approach for the multidimensional case is given by
\begin{equation}\label{eq:unsplit}
\frac{\bQ_{\mu, \bj}^{n+1} - \bQ_{\mu, \bj}^{n}}{ \Delta t^n} = G_{\mu, \bj}(\bQ_h^n) - B_{\mu, \bj}(\bQ_h^n) + \frac{1}{\varepsilon} R_{\mu, \bj}(\bQ_h^{n+1}).
\end{equation}
The implicit discretization of the stiff source term here avoids stability issues and restrictive time increments that come along with a purely explicit scheme. As the source term vanishes in the components corresponding to the relaxed state $\bU$ the scheme~\eqref{eq:unsplit} may be written in an explicit form regarding the updated state $\bQ_{\mu, \bj}^{n+1}$ and thus the implicit term does not increase the computational cost over an explicit scheme. In the uncoupled case the scheme has been shown to be asymptotic preserving in the sense that the discrete relaxation limit yields a consistent scheme for the continuous limit system.

For stability of scheme~\eqref{eq:unsplit} it is sufficient to choose the time increments according to the CFL condition
\begin{equation}\label{eq:cfl}
\Delta t^n = \CFL \min_{i \in \{1,2\}} ~\min_{\mu \in \mathcal{I}_i} ~ \frac{\operatorname{diam}(C_\mu)}{\lambda_i},
\end{equation}
where $\lambda_i$ denotes the maximal diagonal-entry of $\bLambda_i$ and $\CFL \in (0,1]$.

The above scheme yields a first order approximation in time. Higher order discretization is achieved using implicit-explicit Runge-Kutta (IMEX-RK) schemes~\cite{pareschi2005implicexplicrunge} generalizing~\eqref{eq:unsplit}. We restrict the discussion to diagonally implicit IMEX-RK schemes, in which at first $s$ intermediate stages
\begin{equation}\label{eq:imexstages}
  \begin{aligned}
    \bQ_{\mu, \bj}^{n,\nu} =  \bQ_{\mu, \bj}^{n} & + \Delta t \sum_{\ell=1}^{\nu-1} \left( \tilde a_{\nu, \ell} \, G_{\mu, \bj}(\bQ_h^{n,\ell}) - \tilde a_{\nu, \ell} \, B_{\mu, \bj}(\bQ_h^{n,\ell}) +\frac{a_{\nu,\ell}}{\varepsilon}  \, R_{\mu, \bj} (\bQ_h^{n,\ell}) \right) \\
    & + \Delta t \frac{ a_{\nu,\nu}}{\varepsilon} \, R_{\mu, \bj} (\bQ_h^{n,\nu}), \qquad \nu\in\{1,\dots,s\}
  \end{aligned}
\end{equation}
are computed that are afterwards employed in the update formula
\begin{equation}\label{eq:imexupdate}
  \begin{aligned}
    \bQ_{\mu, \bj}^{n+1} =  \bQ_{\mu, \bj}^{n}
    &+ \Delta t \sum_{\nu=1}^{s} \tilde b_\nu \left(  \, G_{\mu, \bj}(\bQ_h^{n,\nu}) - B_{\mu, \bj}(\bQ_h^{n,\nu}) \right)  + \Delta t \sum_{\nu=1}^{s} \frac{b_\nu}{\varepsilon} R_{\mu, \bj} (\bQ_h^{n,\nu}).
  \end{aligned}
\end{equation}

\begin{table}
  \caption{Explicit embedded ($\tilde {\mathbf c}$, $\tilde \bA$, $\tilde{\mathbf b}$, left) and implicit embedded ($\mathbf c$, $\bA$, $\mathbf b$, right) tableau of the second order SSP2(2,2,2) scheme from~\cite{pareschi2005implicexplicrunge} assuming $\gamma=1-\frac{1}{\sqrt{2}}$.}\label{tab:ssp2}
  \centering
  \vspace{5pt}
  \begin{tabular}{>{$}c<{$} | >{$}c<{$} >{$}c<{$}}
    0 & 0 & 0 \\[3pt]
    1 & 1 & 0 \\[3pt] \hline
    \rule{0pt}{1.1\normalbaselineskip}  & \frac 12 & \frac 12
  \end{tabular}
  \hspace{0.2\linewidth}
  \begin{tabular}{>{$}c<{$} | >{$}c<{$} >{$}c<{$}}
    \gamma & \gamma & 0 \\[3pt]
    1-\gamma & 1-2\gamma & \gamma \\[3pt] \hline
    \rule{0pt}{1.1\normalbaselineskip}  & \frac 12 & \frac 12
  \end{tabular}
\end{table}

\begin{table}
  \caption{Explicit embedded ($\tilde {\mathbf c}$, $\tilde \bA$, $\tilde{\mathbf b}$, left) and implicit embedded ($\mathbf c$, $\bA$, $\mathbf b$, right) tableau of the first order unsplit scheme~\eqref{eq:unsplit}.}\label{tab:unsplit}
  \centering
  \vspace{5pt}
  \begin{tabular}{>{$}c<{$} | >{$}c<{$} >{$}c<{$}}
    0 & 0 & 0 \\
    1 & 1 & 0 \\ \hline
      & 1 & 0
  \end{tabular}
  \hspace{0.2\linewidth}
  \begin{tabular}{>{$}c<{$} | >{$}c<{$} >{$}c<{$}}
    0 & 0 & 0 \\
    1 & 0 & 1 \\ \hline
      & 0 & 1
  \end{tabular}
\end{table}

Here, $(\tilde b_\nu)_{\nu=1}^s=\tilde{\mathbf b} \in \R^s$ and $(\tilde a_{\nu,\ell})_{\nu,\ell=1}^s=\tilde \bA \in \R^{s \times s}$ refer to the tableau of the explicit embedded RK scheme  and  $(b_\nu)_{\nu=1}^s=\mathbf b \in \R^s$ and $(a_{\nu,\ell})_{\nu,\ell=1}^s= \bA \in \R^{s \times s}$ to the tableau of the implicit embedded RK scheme. We note that both $\bA$ and $\tilde \bA$ are lower triangular matrices and $\tilde \bA$ has only zero diagonal entries. Similar to standard Runge-Kutta schemes, increasing the number of stages may enhance the order in time. To achieve a specific order in an IMEX-RK scheme, both the usual order conditions for the coefficients of the embedded schemes and additional coupling conditions must be met, see e.g.~\cite{kennedy2003additrungekutta}.

In the context of hyperbolic systems \emph{strong-stability-preserving} (SSP) schemes, which are a family of explicit RK schemes, have been a popular choice as they are capable of preserving properties, such as being total-variation-diminishing or positivity preserving, of the Forward Euler approximation, see~\cite{gottlieb2001stronstabilpreser, gottlieb1998totalrungekutta, chertock2008seconorderposit}. We will adopt the following definition.
\begin{definition}\label{def:imexssp}
  We call the $s$-stage DG-IMEX-RK scheme given by~\eqref{eq:imexstages} and~\eqref{eq:imexupdate} DG-SSP$k$-IMEX scheme iff
  \begin{enumerate}[label=(\alph*)]
    \item the embedded explicit scheme ($\tilde \bA$, $\tilde{\mathbf b}$) is a k-th order SSP-RK scheme and
    \item the embedded implicit scheme ($\bA$, $\mathbf b$) satisfies $a_{k,k} \ne 0$ for all $k\in \{2,\ldots, s\}$ and if $a_{1,1}=0$ the IMEX-RK scheme is globally stiffly accurate, i.e.\ $\tilde b_\nu = \tilde a_{s,\nu}$ and $b_\nu = a_{s,\nu}$ for all $\nu \in \{1,\dots, s\}$.
  \end{enumerate}
\end{definition}

Note that regarding condition (a) in this definition the same SSP-RK scheme may be written in different tableaux by inserting zero lines in $\tilde \bA$ with corresponding zero-entries in $\tilde{\mathbf b}$. An example of an DG-SSP2-IMEX scheme is the 2-stage scheme with implicit and explicit tableau proposed in~\cite{pareschi2005implicexplicrunge} and shown in Table~\ref{tab:ssp2}. Moreover, a simple computation shows that the unsplit scheme~\eqref{eq:unsplit} is formally an DG-SSP1-IMEX scheme with embedded tableaux shown in Table~\ref{tab:unsplit}.

The above tables show the Butcher-tableaux in their common form including the vectors $(c_\nu)_{\nu=1}^s=\tilde{\mathbf c}$ and $(c_\nu)_{\nu=1}^s=\mathbf c$, which satisfy
\[
  c_\nu = \sum_{\ell=1}^s a_{\nu,\ell} \quad \text{and} \quad  \tilde c_\nu = \sum_{\ell=1}^s \tilde a_{\nu,\ell}\quad \text{for }\nu \in \{1,\dots, s\}
  \]
and do not appear in the discretization of autonomous systems such as~\eqref{eq:globrelaxationsystem}.

\subsection{The relaxation limit}\label{sec:relaxationlimit}
In this section we take the relaxation limit $\varepsilon \to 0$ and thereby derive a new scheme for the original coupled problem~\eqref{eq:system}. We adopt the discretization from Sections~\ref{sec:spacedisc} and \ref{sec:timedisc} and start given an $s$-stage DG-SSP$k$-IMEX scheme for the relaxation system~\eqref{eq:globrelaxationsystem}. For the stages~\eqref{eq:imexstages} we denote by $\bQ_h^{n, \nu}$ the DG approximation corresponding to the coefficients and refer to partial states in analogy to~\eqref{eq:Q}.

\begin{lemma}\label{lem:limit} Any $s$-stage DG-SSP$k$-IMEX scheme for the relaxation system~\eqref{eq:globrelaxationsystem} with initial data $\bQ_h^0=(\bU_h^0, V_h^{1,0}, \dots, V_h^{1,0})$, for which the compatibility condition
\begin{equation}\label{eq:visfu0}
  \bV_{\mu, \bj}^{k,0} = \int_{C_\mu} F_i^k(\bU_h^0) \phi_{\mu, \bj} \, dx %V_h^{k,0} = \sum_{\mu \in I} \sum_{\bj \in \mathcal{P}}
\end{equation}
holds for all $\mu \in \mathcal{I}_i$, $\bj \in \mathcal{P}$, $k \in \{1,\dots, d\}$ and $i\in\{1,2\}$
satisfies the analogue condition
\begin{equation}\label{eq:visfustages}
  \bV_{\mu, \bj}^{k,n,\nu} = \int_{C_\mu} F_i^k(\bU_h^{n,\nu}) \phi_{\mu, \bj} \, dx %V_h^{k,0} = \sum_{\mu \in I} \sum_{\bj \in \mathcal{P}}
\end{equation}
for all $\nu\in \{1, \dots, s\}$, $n\in \N$, $\mu \in \mathcal{I}_i$, $\bj \in \mathcal{P}$, $k \in \{1,\dots, d\}$ and $i\in\{1,2\}$ in the relaxation limit as $\varepsilon \to 0$.
\end{lemma}
\begin{remark}
 Condition~\eqref{eq:visfustages} implies that $\bV_h^{k,n, \nu}$ is an $L^2$-projection of $\bF^k_1(\bU_h^{n,\nu})$ and $\bF^k_2(\bU_h^{n,\nu})$ on the DG spaces $\mathcal{S}_1^{m_1}$ and $\mathcal{S}_2^{m_2}$. Moreover, if $a_{1,1}\ne0$ the compatibility condition~\eqref{eq:visfu0} is not necessary for condition~\eqref{eq:visfustages} to hold.
\end{remark}
\begin{proof}[Proof of Lemma~\ref{lem:limit}]
 We first introduce the matrix
  \[
    \mathbf P^k_i =
    \begin{pmatrix}
      \bzero & \delta_{1,k} \bI & \cdots & \delta_{d,k} \bI
    \end{pmatrix} \in \R^{m_i \times m_i(d+1)}
  \]
  and then multiply~\eqref{eq:imexstages} by $\mathbf P^k_i$ to obtain
  \begin{equation}\label{eq:stageV}
    \begin{aligned}
    \bV_{\mu, \bj}^{k, n, \nu} =  \bV_{\mu, \bj}^{k,n} & +  \Delta t  \mathbf P^k_i \sum_{\ell=1}^{\nu-1} \left( \tilde a_{\nu, \ell} \, G_{\mu, \bj}(\bQ_h^{n,\ell}) - \tilde a_{\nu, \ell} \, B_{\mu, \bj}(\bQ_h^{n,\ell}) +\frac{a_{\nu,\ell}}{\varepsilon}  \, R_{\mu, \bj} (\bQ_h^{n,\ell}) \right) \\
                             & + \Delta t \frac{ a_{\nu,\nu}}{\varepsilon} \,  \mathbf P^k_i R_{\mu, \bj} (\bQ_h^{n,\nu})
    \end{aligned}
  \end{equation}
  to account for stage $\nu$ of the IMEX-RK scheme with respect to the variable $\bV^k$. If $a_{1,1} \ne 0$ multiplying~\eqref{eq:stageV} by $\varepsilon$ we get
  \begin{equation}\label{eq:Ris0}
    \bzero = \mathbf P^k_i R_{\mu, \bj} (\bQ_h^{n,\nu}) = \int_{C_\mu} \left( \bF_i^k(\bU_h^{n, \nu}) - \bV_h^{k,n,\nu} \right) \phi_{\mu, \bj} \, dx
  \end{equation}
  in the limit $\varepsilon \to 0$ for all $\nu \in\{1,\dots, s\}$ by induction over $\nu$. By orthogonality~\eqref{eq:Ris0} is equivalent to~\eqref{eq:visfustages} the statement follows in this case.

  If $a_{1,1}=0$ and the IMEX-RK scheme is globally stiffly accurate it holds for all $n\in\N$ that
  \[
    \bQ_h^{n,s} = \bQ_h^{n+1} = \bQ_h^{n+1,1}.
  \]
  Thus, from~\eqref{eq:visfu0} it follows~\eqref{eq:visfustages} and thereby also~\eqref{eq:Ris0} for $n=2$ and $\nu=1$ in the limit $\varepsilon \to 0$. By the above argument~\eqref{eq:Ris0} holds also for $n=2$ and $\nu\in \{2, \dots, s\}$. Eventually,~\eqref{eq:visfustages} follows for all $n \in \N$ and $\nu\in \{1, \dots, s\}$ by induction over $n$.
  \end{proof}

  Lemma~\ref{lem:limit} allows for a new scheme formulation of any DG-SSP$k$-IMEX scheme for the relaxation system~\eqref{eq:globrelaxationsystem} in the relaxation limit. To this end for $\bU_h \in \mathcal{S}^{m_i}$ let
\begin{equation}
  \mathcal{F}_i^k[\bU_h] = \sum_{\mu \in \mathcal{I}} \sum_{\mathbf j \in \mathcal{P}} \int_{C_\mu} \bF_i^k(\bU_h) \phi_{\mu, \bj} \, dx ~ \phi_{\mu, \bj}(\bx),
\end{equation}
denote the $L^2$-projection of $\bF_i^k(\bU_h)$ on $\mathcal{S}^{m_i}_i$ for $i \in \{1,2\}$. Neglecting the update of the relaxation variable, the semi-discrete scheme in the relaxation limit takes the form
\begin{align}\label{eq:DGUsemidiscrete}
  \frac{\partial \bU_{\mu, \bj}}{\partial t}  &= \int_{C_\mu} \sum_{k=1}^d  \mathcal{F}_i^k[\bU_h] \frac{\partial \phi_{\mu, \bj}}{\partial x_j} \, dx  - \int_{\partial C_\mu} \hat{\mathbf g}_\mu(\bU_h^-, \bU_h^+; \bn) \phi_{\mu, \bj} \, dS \notag\\ 
  &\coloneqq H_{\mu, \bj}(\bU_h) - C_{\mu, \bj}(\bU_h).
\end{align}
For any $\bx \in \partial C_\mu$ residing on the boundary of two cells with $\bx \notin \Gamma$ and $\mu\in \mathcal{I}_i$ the numerical flux takes the form
\begin{equation}\label{eq:numuflux}
  \hat{\mathbf g}_\mu (\bU_h^-, \bU_h^+; \bn) = \frac 12 \bSigma_i (\bU_h^- - \bU_h^+) + \frac12 \left( \mathcal{F}_i^{\bn}[\bU_h](\bx^-) + \mathcal{F}_i^{\bn}[\bU_h](\bx^+) \right),
\end{equation}
where we have employed the notations
\begin{equation}\label{eq:VFhn}
 \mathcal{F}_i^{\bn}[\bU_h] = \sum_{j=1}^d n_j \mathcal{F}_i^{j,\pm}[\bU_h]\quad \text{and} \quad \bSigma_i = \sqrt{\sum_{j=1}^d n_j^2 (\bLambda_i^j)^2}
\end{equation}
and $\bx^\mp$ refer to the inner and outer limit in $\bx$ with respect to the cell $C_\mu$.

Next, we consider the case, where $\hat \bx$ resides on the boundary of the cells $C_\mu\subset \Omega^C_1$ and $C_{\tilde \mu} \subset \Omega^C_2$ and thus $\hat \bx \in \Gamma$. In analogy to \eqref{eq:interfacetraces} let $\bU_h^-$ and $\bU_h^+$ denote the interface traces in $\hat \bx$ and $\bn$ the normal in $\hat \bx$ with respect to $\Gamma$ pointing towards $\Omega_2^C$. Then the fluxes are given by
\begin{equation}\label{eq:numfluxuinterface}
  \begin{aligned}
    \hat{\mathbf g}_\mu (\bU_h^-, \bU_h^+; \bn) &= \frac 12 \bSigma_1 (\bU_h^- - \bU_R) + \frac12 \left( \mathcal{F}_1^{\bn}[\bU_h](\hat \bx^-) + \bV_R^\bn \right),\\
    \hat{\mathbf g}_{\tilde \mu} (\bU_h^+, \bU_h^-; -\bn) &= \frac 12 \bSigma_2 (\bU_h^+ - \bU_L) - \frac12 \left( \mathcal{F}_2^{\bn}[\bU_h](\hat \bx^+) + \bV_L^\bn \right)
  \end{aligned}
\end{equation}
with coupling states defined by the RS from Section~\ref{sec:riemann} and
\begin{equation}\label{eq:rsu}
  (\bU_R, \bV_R^\bn, \bU_L, \bV_L^\bn) = \rs(\bU_h^-,  \mathcal{F}_i^{\bn}[\bU_h](\hat \bx^-), \bU_h^+,  \mathcal{F}_i^{\bn}[\bU_h](\hat \bx^+)).
\end{equation}
The states $\bV_R^\bn$ and $\bV_L^\bn$ that appear in~\eqref{eq:numfluxuinterface} and~\eqref{eq:rsu} are computable from $\bU_h$ and do \emph{not} require a discretization of an auxiliary relaxation variable.

Regarding the time discretization the DG-SSP$k$-IMEX scheme given by \eqref{eq:imexstages} and \eqref{eq:imexupdate} recovers in the relaxation limit the purely explicit DG-SSP$k$-RK scheme with stages
\begin{equation}\label{eq:sspstages}
  \begin{aligned}
    \bU_{\mu, \bj}^{n,\nu} =  \bU_{\mu, \bj}^{n} & + \Delta t \sum_{\ell=1}^{\nu-1} \left( \tilde a_{\nu, \ell} \, H_{\mu, \bj}(\bU_h^{n,\ell}) - \tilde a_{\nu, \ell} \, C_{\mu, \bj}(\bU_h^{n,\ell})  \right),  \qquad \nu\in\{1,\dots,s\}
  \end{aligned}
\end{equation}
and update formula
\begin{equation}\label{eq:sspupdate}
  \begin{aligned}
    \bU_{\mu, \bj}^{n+1} =  \bU_{\mu, \bj}^{n}
    &+ \Delta t \sum_{\nu=1}^{s} \tilde b_\nu \left(  \, H_{\mu, \bj}(\bU_h^{n,\nu}) - C_{\mu, \bj}(\bU_h^{n,\nu}) \right).
  \end{aligned}
\end{equation}
In particular, the first and second order IMEX schemes in Tables~\ref{tab:unsplit} and \ref{tab:ssp2} recover the forward Euler and the SSP2 scheme, respectively. We summarize the analysis of this section in the following theorem.

\begin{theorem}\label{thm:limit}
In the relaxation limit $\varepsilon \to 0$ the DG-SSP$k$-IMEX scheme for the coupled relaxation system~\eqref{eq:globrelaxationsystem} given by~\eqref{eq:DGsemidiscrete}, \eqref{eq:imexstages} and \eqref{eq:imexupdate} recovers a DG-SSP-RK scheme for the original coupled system~\eqref{eq:system} that is given by~\eqref{eq:DGUsemidiscrete}, \eqref{eq:sspstages} and \eqref{eq:sspupdate}.
\end{theorem}
\begin{remark}\label{rem:lxf}
  If we ignore the coupling at the interface, restrict the polynomial order taking $p=1$, simplify the matrices of the relaxation system assuming both \eqref{eq:samelambdai} and \eqref{eq:relspeed} and choose the unsplit scheme~\eqref{eq:unsplit} for the time discretization the classical Lax-Friedrichs scheme on the subdomains $\Omega_1^C$ and $\Omega_2^C$ is obtained in the limit $\varepsilon \to 0$.
\end{remark}

\subsection{The implementation}\label{sec:implementation}
In our numerical computations we approximate the solution of the coupled problem by piecewise quadratic polynomials ($p=3$). We only compute the numerical solution in the relaxation limit employing the scheme derived in Section~\ref{sec:relaxationlimit} and use the SSP3 method to discretize in time. To decrease the numerical viscosity, we adopt local Lax--Friedrichs numerical fluxes away from the interface and thus compute the velocities within the matrix $\bSigma_i$ in \eqref{eq:numuflux} based on the local states $\bU_h^-$ and $\bU_h^+$. For the flux computation at the interface according to \eqref{eq:numfluxuinterface} the matrices $\bSigma_1$ and $\bSigma_2$ are chosen such that they maximize those velocities with regard to the corresponding domain in every time step. Taking adaptive time increments with Courant number $\text{CFL}=0.7$ numerical stability has been verified in our computations.

Moreover, we employ the Shu limiter within our DG scheme to ensure a total variation bound of the numerical solution \cite{cockburn1989tvbrungekuttagaler}. In the solution update a five-point Gauss quadrature is used. For efficiency we rely on an adaptive mesh refinement (AMR) approach based on a multiresolution analysis using multiwavelets~\cite{hovhannisyan2014adaptgaler}. Further details on this strategy and its implementation are given in~\cite{gerhard2016adaptgaler, gerhard2022}. 

\section{A fluid-structure coupling example}\label{sec:example}
For a proof of concept we next apply our approach to the coupling of a linear-elastic structure with an inviscid fluid. To this end, the coupling problem is specified in Section \ref{sec:FSI-problem}, a relaxation based Riemann solver is constructed in Sections~\ref{sec:FSI-relaxation} and~\ref{sec:FSI-RS} and numerical results are presented in Section~\ref{sec:FSI-results}.

\subsection{The fluid-structure coupling model}\label{sec:FSI-problem}
In this section we briefly summarize the two coupled models and provide coupling conditions. While the coupling problem is here introduced for general $d\in\{1,2, 3\}$, in our numerical computations in Section~\ref{sec:FSI-results} we focus on the case of a two-dimensional spatial domain. 

\paragraph{The linear-elastic solid model}
We adopt the elastic structure model
\begin{equation}\label{eq:elastic}
  \left\{
    \begin{aligned}
      \ddt \bw &- \frac{1}{\rho_s} \nabla \cdot \sigma = 0,\\
      \ddt \mathbf \sigma &- \lambda (\nabla \cdot \bw)I - \mu (\nabla \bw + \nabla \bw^T)= 0
    \end{aligned}\right. \qquad \qquad (t,\bx) \in (0,\infty) \! \times \Omega_1
\end{equation}
for the deformation velocity $\bw=(w_1,\dots, w_d)^T$ and the symmetric shear stress tensor $\sigma = (\sigma_{i,j})_{i,j=1}^d =\sigma^T\in \R^{d \times d}$ of a solid. Model parameters are the material density $\rho_s$ and the Lamé constants $\lambda, \mu >0$, which relate to the dilatation wave velocities $c_1, c_2> =0$ by $c_1^2 = \frac{2}{\rho_s}(\mu + \lambda)$ and $c_2^2= \frac{\mu}{\rho_s}$.

The structure model is based on the continuum mechanics conservation principles for mass, momentum, angular momentum, and energy. The following additional assumptions are made: (i) the stress tensor is symmetric, (ii) temperature is constant, (iii) density variations in the solid are significantly smaller than in the liquid, (iv) displacements in the solid are negligible, and (v) the solid material is homogeneous and isotropic. For further details on the linear-elastic model and its assumptions, refer to~\cite{dickopp2013coupl, herty2018fluid}.

Letting model~\eqref{eq:elastic} account for the dynamics on $\Omega_1$ of the coupled problem~\eqref{eq:system} and dropping the redundant entries of $\sigma$, the state space is of dimension $m_1=\frac{3}{2}d + \frac{d^2}{2}$ and the flux functions can be expressed as $\bF_1^j(\bU) = \bA_1^j\, \bU$ for $j\in\{1.\dots,d\}$. In particular, for $d=2$ the state takes the form $\bU = (w_1,w_2,\sigma_{11},\sigma_{12},\sigma_{22})^T$ and taking $\alpha= 1 - 2 (c_2/c_1)^2$ the flux functions are determined by
\begin{equation}
  \bA_1^1 =
   - \rho_s^{-1}\,
   \begin{pmatrix}
     0 & 0 & 1 & 0 & 0 \\
     0 & 0 & 0 & 1 & 0 \\
    \rho_s^2 c_1^2 & 0 & 0 & 0 & 0 \\
     0 &\rho_s^2 c_2^2 & 0 & 0 & 0 \\
    \alpha\,\rho_s^2 c_1^2 & 0 & 0 & 0 & 0
   \end{pmatrix},\quad
\bA_1^2 =
   - \rho_s^{-1}\,
   \begin{pmatrix}
     0 & 0 & 0 & 1 & 0 \\
     0 & 0 & 0 & 0 & 1 \\
     0 & \alpha\,\rho_s^2 c_2^2 & 0 & 0 & 0 \\
    \rho_s^2 c_2^2 & 0 & 0 & 0 & 0 \\
     0 & \rho_s^2 c_1^2  & 0 & 0 & 0
   \end{pmatrix} .
 \end{equation}

\paragraph{The compressible fluid model}  To model fluid flow we employ the inviscid compressible Euler equations, which read
\begin{equation}\label{eq:fluid}
  \left\{
  \begin{aligned}
          \ddt \rho + \nabla \cdot (\rho \, \bv) &= 0, \\[5pt]
      \ddt (\rho \bv) + \nabla \cdot (\rho \, \bv \otimes \bv + p I) &= 0, \\[5pt]
      \ddt (\rho \,E) + \nabla \cdot \left(\rho \, \bv \left(E+ \frac{p}{\rho}\right) \right)  &= 0
\end{aligned}\right. \qquad \qquad (t,\bx) \in (0,\infty) \! \times \Omega_2
\end{equation}
for the density $\rho$, the momentum $\rho\,\bv$ and the total energy $\rho\,E$.
By $\bv=(v_1,\dots, v_d)^T\in \R^d$ and $p$ we refer to the velocity and the pressure of the fluid. The fluid is assumed to be a stiffened gas, meaning that the equation of state
\begin{align}
\label{eq:EoS}
 p = (\gamma-1)\,\rho\, e - \gamma\,\pi 
\end{align}
holds, where $\gamma>1$ and $\pi \ge 0$ denote the ratio of heat capacities and the pressure stiffness.
The internal energy $e$ appearing in \eqref{eq:EoS} is related to the total energy by
\begin{align}
\label{eq:rhoE}
  E = e + \frac{1}{2} |\bv|^2 .
\end{align}
We call a state of the fluid model~\eqref{eq:fluid} \emph{physically admissible} if the density is positive and the pressure satisfies the bound $p \geq -\pi$.

Imposing the fluid model on $\Omega_2$ of the coupled problem~\eqref{eq:system} we note that the space state is of dimension $m_2=d + 2$. In the case $d=2$ the state reads $\bU = (\rho,\rho\,v_1,\rho\,v_2,\rho\,E)^T$ and the fluxes $\bF^1_2$ and  $\bF^2_2$ are given by
\begin{equation}\label{eq:fluid-2d}
  \begin{aligned}
    \bF_2^1(\bU) &=  (\rho\,v_1,\rho\,v_1^2 + p,\rho\,v_1\,v_2,  v_1\,(\rho \, E + p))^T, \\
    \bF_2^2(\bU) &=  (\rho\,v_2,\rho\,v_1 v_2, \rho\,v_2^2 + p, v_2\,(\rho \, E + p))^T.
  \end{aligned}
\end{equation}

\paragraph{Coupling condition}
At the interface $\Gamma$ we follow \cite{dickopp2013coupl} and impose the equality of normal velocities, i.e., 
\begin{align}\label{eq:cplvelocity}
  \bn^T\, \bw(t, \bx^-)  = \bn^T \,\bv(t, \bx^+)  \qquad \text{for a.e. } t \geq 0 \quad \text{and }\bx \in \Gamma.
\end{align}
In addition, the solid stress and the fluid pressure should be related by 
\begin{align}
  \label{eq:cplstress}
 \bn^T \sigma(t, \bx^-) \bn = -p(t, \bx^+)  \qquad \text{for a.e. } t \geq 0 \quad \text{and }\bx \in \Gamma.
\end{align}
Consequently, using the notations $w_\bn\coloneqq \mathbf n^T\, \bw$, $v_\bn \coloneqq \mathbf n^T\, \bv$ and $\sigma_\bn\coloneqq \bn^T \sigma \bn$ the coupling function~\eqref{eq:couplingu} for our problem reads
\begin{align}
\label{eq:coupled-system-multid-coupling-u}
\Psi^\bn_U(\bU(t, x^-),\bU(t, x^+)) \coloneqq
\begin{pmatrix}
  w_\bn-\frac{\rho\,v_\bn}{\rho}\\[5pt]
  \sigma_\bn +(\gamma-1)\left(\rho\,E- \frac{|\rho \bv |^2}{2 \rho} \right) - \gamma\, \pi
\end{pmatrix} ,
\end{align}
where for readability we drop the time and space dependence on the right-hand side.

\subsection{The coupled fluid-structure relaxation system}\label{sec:FSI-relaxation}

In this section we consider a Jin-Xin type relaxation of the coupling problem introduced in Section~\ref{sec:FSI-problem}. Due to the linearity of the structure model, relaxation is not required to handle the coupling of the subproblem in $\Omega_1$. Relaxing only the fluid part of the model, we avoid increasing the number of variables in the structure model and thereby reduce the total number of coupling conditions needed. Thus, the global relaxation system~\eqref{eq:globrelaxationsystem} and the localized relaxation system given by~\eqref{eq:relaxationprojected} and~\eqref{eq:relaxationprojected2} simplify. Considering the transformation~\eqref{eq:transform} with respect to an arbitrary interface point $\hat x \in \Gamma$ and neglecting the tilde the latter takes the form
   \begin{equation}\label{eq:FSI-relaxationprojected}
 \left\{
   \begin{aligned}
      \pdiff{\bU}{t} + A_1^\bn\,\pdiff{\bU}{x_\bn} &= \bzero \qquad
      &\text{for }(t, \bx) \in (0, \infty) \times \R^d_-\\
     \pdiff{\bU}{t} + \pdiff{\bV^\bn}{x_\bn} &= \bzero \qquad
      &\text{for }(t, \bx) \in (0, \infty) \times \R^d_+,\\
     \pdiff{\bV^\bn}{t} + (\bLambda_2^\bn)^2 \, \pdiff{\bU}{x_\bn} &= \frac{1}{\varepsilon} \left( \bF^\bn_2(\bU) - \bV^\bn \right)  \qquad &\text{for }(t, \bx) \in (0, \infty) \times \R^d_+,
   \end{aligned}
 \right.
\end{equation}
where
\begin{equation}\label{eq:FSI-fluxes}
  A_1^\bn = \sum_{j=1}^d A_1^j\,n_j ,\quad
  \bF^\bn_2(\bU)  = \sum_{j=1}^d \bF_2^j(\bU)\,n_j.
\end{equation}
The auxiliary variables connected to the tangential directions $\bt_1, \dots, \bt_{d-1}$ are governed by the ordinary differential equations
\begin{equation}
\label{eq:FSI-relaxationprojected2}
    \begin{aligned}
      \pdiff{\bV^{\bt_k}}{t}  &= \frac{1}{\varepsilon} \left(\bF^{\bt_k}_2(\bU) - \bV^{\bt_k} \right)  \qquad &&\text{for }(t, \bx) \in (0, \infty) \times \R^d_+
\end{aligned}
  \end{equation}
  The localized relaxation system \eqref{eq:FSI-relaxationprojected} is closed by coupling condition~\eqref{eq:locrelcoupling} with $\tilde \bQ^\bn(t, \bx) = \bU(t, \bx)$ for $\bx\in\Omega_1$ and $\tilde \bQ^\bn(t, \bx) = (\bU(t, \bx),\bV^n(t, \bx))$ for $\bx\in\Omega_2$ and the vector-valued map $\Psi_Q^\bn : \mathcal{D}_1  \times \mathcal{D}_2 \times \R^{m_2} \to \R^{\tilde \ell}$.
  The number of coupling conditions ${\tilde \ell}$ depends on the left-propagating waves of the linear-elastic model and the right-propagating waves of the relaxed fluid model. While by our analysis in Section~\ref{sec:halfRP} the relaxed fluid model accounts for $m_2$ right-propagating waves, the number of relevant waves of the structure model is determined by the eigenvalues of the flux Jacobian $\bA_1^\bn$ in~\eqref{eq:FSI-fluxes}.

  \paragraph{The relaxed fluid-structure coupling condition in 2D} In the following we design the relaxed coupling condition~\eqref{eq:relcoupling} corresponding to~\eqref{eq:coupled-system-multid-coupling-u} in the case $d=2$. Without loss of generality we assume the normal $\bn = (1, 0)^T$ and thus have $w_\bn = w_1$, $\sigma_n=\sigma_{11}$ and $v_\bn=v_1$. In the structure model
the eigenvalues of $A_1^\bn = A_1^1$ and their corresponding right eigenvectors are
\begin{equation}\label{eq:eigenstructure-linelast-2d}
  \begin{aligned}
  & \lambda_{1,1\pm} = \pm c_1,\quad
  \lambda_{1,2\pm} = \pm c_2,\quad
  \lambda_{1,0}=0, \\
  & \br_{1\pm} =  (1,0,\mp\rho\, \c_1,0,\mp\alpha\,\rho\,c_1)^T,\
    \br_{2\pm} =   (0, 1, 0 ,\mp \rho\, c_2,0)^T,\
    \br_{0} =  (0, 0, 0 ,0,-1/\beta)^T
  \end{aligned}
\end{equation}
with $\beta \coloneqq c_1^2-c_2^2 > 0$, see \cite{herty2018fluid} for details.
Eigenvalues and corresponding right eigenvectors for the relaxed fluid model are taken from~\eqref{eq:diagonalized}. We conclude that the coupling states are determined by six relevant waves. As condition~\eqref{eq:coupled-system-multid-coupling-u} is independent of the quantities $v_2$ and $\sigma_{12}$ and thus of the wave corresponding to eigenvalue $\lambda_{1,2-}$ and eigenvector $\br_{2-}$ the number of coupling conditions is chosen according to the remaining number of relevant waves ${\tilde \ell} = 5$. Hence, in addition to the two conditions within~\eqref{eq:coupled-system-multid-coupling-u} three additional ones are needed that become redundant in the relaxation limit. Following Alg.~\ref{algo:rsconstruction}, Step 1 and 2, we choose
\begin{align}\label{eq:coupling-q-fsi}
\Psi_Q^\bn(\bU^-,(\bU^+,\bV^+)) \coloneqq
\begin{pmatrix}
  w_1^-- \frac{[\rho\,v_1]^+}{\rho^+}\\[5pt]
  \sigma_{11}^- + (\gamma-1)\left( [\rho\,E]^+ - \frac{ ([\rho v_1]^+)^2 + ([\rho v_2]^+)^2}{2 \rho^+} \right) - \gamma\, \pi\\[5pt] 
   \frac{[V^{\rho}]^+}{\rho^+} - w_1^- \\[5pt]
 \frac{([\rho\,v_1]^+)^2}{\rho^+}-[V^{\rho\,v_1}]^+ - \sigma_{11}^-\\[5pt]
  [V^{\rho\,v_2}]^+ - [\rho\,v_2]^+\,  w_1^-
\end{pmatrix},
\end{align}
where we employ the notation $\bV = (V^{\rho}, V^{\rho\,v_1}, V^{\rho\,v_2}, V^{\rho\,E})^T$ to refer to the components of the auxiliary variable. A simple computation shows that
$\Psi_Q(\bU^-,(\bU^+,\bF_2^\bn(\bU^+)))= \bzero$ iff $\Psi_U\left(\bU^-,\bU^+\right) = \bzero$, i.e., the relaxed coupling condition~\eqref{eq:coupling-q-fsi} is consistent with the original coupling condition~\eqref{eq:coupled-system-multid-coupling-u} according to Def.~\ref{def:consistent}.

\subsection{The fluid-structure Riemann solver}\label{sec:FSI-RS}
In this section we derive the RS for the coupled relaxation system~\eqref{eq:FSI-relaxationprojected} in the case $d=2$ given the coupling condition~\eqref{eq:coupling-q-fsi}. We make the simplification~\eqref{eq:relspeed} within the relaxation system and choose $\bLambda_2^1 = \bLambda_2^2 = \lambda_2 I$ for a sufficiently large relaxation speed $\lambda_2$ for the relaxed Euler equations.

Since only the fluid model has been relaxed, the RS is a map from $\mathcal{D}_1\times \mathcal{D}_2 \times \R^4$ to itself,
  \begin{equation}\label{eq:rs-fsi}
    \rs: (\bU^-,\bU^+, \bV^+) \mapsto (\bU_R,  \bU_L, \bV_L).
  \end{equation}
  In the following we use an analogue notation for the state components as in~\eqref{eq:coupling-q-fsi}.
  To construct the coupling data $\bU_R, \bU_L$ and $\bV_L$ from the trace data $\bU^-, \bU^+, \bV^+$ we follow step~3 of Algorithm~\ref{algo:rsconstruction} by solving the associated two half-Riemann problems, in which the states $\bU_R$ and $(\bU_L,\bV_L)$ are located on Lax curves of negative and positive speeds of the structure model and the fluid model, respectively, and coupled by condition~\eqref{eq:coupling-q-fsi}. The half-Riemann problems imply the conditions
  \begin{equation}\label{eq:lax-fsi}
    \bU_R = \bU^- - \Sigma_1 \, \br_{1-} - \Sigma_2 \, \br_{2-}  \quad\text{and}\quad \bV_L- \bV^+ = \lambda_2 (\bU_L - \bU^+)
  \end{equation}
  for two parameters $\Sigma_1$ and $\Sigma_2 \in \R$. The first condition follows from~\eqref{eq:eigenstructure-linelast-2d}. As argued in Section~\ref{sec:FSI-relaxation} our coupling condition is independent of the wave corresponding to eigenvalue $\lambda_{1,2-}$ and eigenvector $\br_{2-}$ and thus \eqref{eq:coupling-q-fsi} cannot determine $\Sigma_2$. By convention and motivated by the idempotency of the RS we choose $\Sigma_2=0$ and thus obtain
\begin{equation}\label{eq:w2r}
  [w_2]_R = w_2^-\quad \text{and} \quad [\sigma_{12}]_R=\sigma_{12}^-.
\end{equation}
The second condition in~\eqref{eq:lax-fsi} for the relaxed liquid model follows from~\eqref{eq:laxrelation} and allows us to express the variable $\bV_L$ in terms of $\bU_L$. Substituting in the required coupling condition $\Psi_Q^\bn(\bU_R,(\bU_L,\bV_L)) = \bzero$ we obtain the system
\begin{equation}\label{eq:fsi-nonlinearsystem}
\left\{
  \begin{aligned}
     w_1^- - \Sigma_1 - \frac{[\rho\,v_1]_L}{\rho_L} &=0,\\[3pt]
  \sigma_{11}^-  - \Sigma_1 \,\rho_s c_1 + (\gamma-1)( [\rho\,E]_L - \frac{([\rho\,v_1]_L)^2  + ([\rho\,v_2]_L)^2}{2\,\rho_L} - \gamma\, \pi &=0,\\[3pt]
  \frac{[V^\rho]^+ + (  \rho_L - \rho_0^+ ) \lambda_2}{\rho_L}-  w_1^- + \Sigma_1 &=0,\\[3pt]
  \frac{[\rho\,v_1]_L^2}{\rho_L} - [V^{\rho\,v_1}]^+ - (  [\rho\,v_1]_L - [\rho\,v_1]^+ ) \lambda_2  - \sigma_{11}^- + \Sigma_1 \, c_1 \rho_s &=0,\\[3pt]
 [V^{\rho\,v_2}]^+ + (  [\rho\,v_2]_L - [\rho\,v_2]^+ ) \lambda_2 + \left( \Sigma_1 - w_1^- \right) [\rho\,v_2]_L&=0
  \end{aligned}\right.
\end{equation}
with unknown variables $\Sigma_1$, $\rho_L$, $[\rho\,v_1]_L$, $[\rho\,v_2]_L$ and $[\rho\,E]_L$. We identify a solution by the following steps:
 first we employ the first equation of~\eqref{eq:fsi-nonlinearsystem} to express $\Sigma_1$ in terms of $\rho_L$ and $[\rho\,v_1]_L$, i.e., $\Sigma_1(\rho_L, [\rho\,v_1]_L)$. Substituting the latter into equations two, three and five of~\eqref{eq:fsi-nonlinearsystem} we obtain explicit expression for the variables $[\rho\,E]_L= [\rho\,E]_L(\rho_L, [\rho\,v_1]_L,[\rho\,v_2]_L)$, $[\rho\,v_1]_L=[\rho\,v_1]_L(\rho_L)$ and $[\rho\,v_2]_L=[\rho\,v_2]_L(\rho_L,[\rho\,v_1]_L)$. Finally we substitute our expressions for $\Sigma_1$ and $[\rho\,v_1]_L$ in equation four of~\eqref{eq:fsi-nonlinearsystem} and multiply by $\rho_L$, which assuming a physically admissible coupling state is positive. The result is a linear equation in $\rho_L$ with an explicitly computable solution. The remaining variables are determined by backward substitution. We obtain, in particular,
  \begin{equation}\label{eq:RS-euler}
    \begin{aligned}
      \rho_L   & = \frac{\rho^+(\lambda_2  - v_1^+) (\rho^+ (\lambda_2 -  v_1^+) + c_1\rho_s )}{(\lambda_2  -  v_1^+)^2\rho^+ + (\lambda_2  - w_1^-) c_1\rho_s + \sigma_{11}^-  + p^+},\\
  [\rho \, v_1]_L & = [\rho v_1]^+ + (  \rho_L - \rho^+) \lambda_2, \qquad  [\rho_L \, v_2]_L  =\rho^+ v_2^+ \frac{ \lambda_2- v_1^+}{\lambda_2-[v_1]_L}, \\
     [\rho\, E]_L & =   \frac{(w_1^- - [v_1]_L) \rho_s c_1 + \gamma\, \pi -\sigma_{11}^- }{\gamma-1} + \frac{1}{2} ( [\rho v_1]_L^2 + [\rho v_2]_L^2)  
\end{aligned}
\end{equation}
for the coupling state $\bU_L$ and further, inserting the explicit form of $\Sigma_1$ in the first condition of~\eqref{eq:lax-fsi} and using~\eqref{eq:w2r},
\begin{equation}\label{eq:RS-structure}
  \begin{aligned}
    [w_1]_R  &=  [v_1]_L, & [w_2]_R  &=  w_2^-, \qquad [\sigma_{11}]_R = \sigma_{11}^- - (w_1^- - [v_1]_L) \rho_s c_1,\\
 [\sigma_{12}]_R & = \sigma_{12}^- , &[\sigma_{22}]_R  &=  \sigma_{22}^- - (w_1^- - [v_1]_L) \alpha \rho_s c_1 .
  \end{aligned}
\end{equation}
for the coupling state $\bU_R$, determining together with the second condition in \eqref{eq:lax-fsi} the RS~\eqref{eq:rs-fsi}.

\begin{proposition}\label{prop:physical-states}
  Suppose that the relaxation speed $\lambda_2$ is such that the subcharacteristic condition~\eqref{eq:subcharacteristic} holds, and $\bU^+$ is a physically admissible state of the fluid model~\eqref{eq:fluid} satisfying the conditions
\begin{equation}
  \begin{aligned}\label{eq:couplingbound}
    v_1^+ + c^+ > w_1^+, \qquad \sigma_{11}^-> -(p^+ + [c^+]^2 \rho^+), \\
    (p^+- \pi) (\rho_s c_1 + c^+ \rho^+) \geq (\lambda_2-v_1^+)  c^+ \left( \sigma_{11}^- + p^+ + v_1^+ - w_1^- \right)
  \end{aligned}
\end{equation}
with $c^+ = \sqrt{\gamma \,\frac{p^+ + \pi}{\rho^+}}$ referring to the speed of sound. Then $\bU_L$ is a physically admissible state of~\eqref{eq:fluid}.  
\end{proposition}
\begin{proof}
  The subcharacteristic condition~\eqref{eq:subcharacteristic} and assumption~\eqref{eq:relspeed} imply that $\lambda_2 \geq |v_1^+| + c^+$. Therefore, the numerator of $\rho_L$ in \eqref{eq:RS-euler} is clearly positive. The same is true for the denominator, since we have 
  \[
    (\lambda_2  -  v_1^+)^2\rho^+ + (\lambda_2  - w_1^-) c_1\rho_s + \sigma_{11}^-  + p^+ \geq [c^+]^2 \rho^+ + (v_1^+ - w_1^- + c^+) c_1 \rho_s + \sigma_{11}^- + p^+
  \]
  as well as $v_1^+ - w_1^- > -c ^+$ and $\sigma_{11}^- + p^+ > -[c^+]^2 \rho^+$ thanks to~\eqref{eq:couplingbound}. This shows that $\rho_L>0$.

  Next, we prove the lower bound of the pressure $p_L$.  Comparing $[\rho \, E]_L$ in \eqref{eq:RS-euler} to \eqref{eq:rhoE} we obtain $e_L$ and thus using~\eqref{eq:EoS} and $[v_1]_L=[w_1]_R$ the pressure reads
  \begin{equation}\label{eq:pL}
    p_L = (\gamma - 1) \rho_L e_L - \gamma \pi = (w_1^- - [w_1]_R) \rho_s c_1  -\sigma_{11}^-. 
  \end{equation}
  Employing our solution of system~\eqref{eq:fsi-nonlinearsystem} we note that
  \begin{equation}\label{eq:sigma1}
     w_1^- - [w_1]_R  = \Sigma_1 =  \frac{  \rho^+ (w_1^- - v_1^+) ( \lambda_2  -  v_1^+)   + \sigma_{11}^- + p^+}{ \rho^+ (\lambda_2  -  v_1^+) + \rho_s c_1} %\geq  \frac{  c^+ \,\rho^+ (w_1^- - v_1^+)  + \sigma_{11}^- + p^+}{ \rho^+ (\lambda_2  -  v_1^+) + \rho_s c_1}.
   \end{equation}
   Since $D \coloneqq \rho^+ (\lambda_2  -  v_1^+) + \rho_s c_1>0$ the bound $p_L\geq - \pi$ follows from
   \[
     D(p_L + \pi) =  (p^+- \pi) (\rho_s c_1 + (\lambda_2- v_1^+) \rho^+) - (\lambda_2-v_1^+) c^+ \left( \sigma_{11}^- + p^+ + v_1^+ - w_1^- \right) \geq 0,
     \]
     where we have used the subcharacteristic condition and the third inequality in assumption~\eqref{eq:couplingbound}.
   \end{proof}

\begin{remark}{(Properties of the RS and the coupling states)}\label{rem:prop-coupling-states}
  \begin{enumerate}
  \item The technical conditions in~\eqref{eq:couplingbound} assumed in Proposition~\ref{prop:physical-states} follow if $p^+>-\pi$ and $\| \Psi^\bn_U(\bU^-, \bU^+)\|$ is sufficiently small, i.e. the trace data satisfies coupling condition~\eqref{eq:coupling-q-fsi} up to a small error. The latter is expected during numerical computations if we assume that the coupling condition \eqref{eq:coupled-system-multid-coupling-u} is satisfied with respect to the initial data and the numerical scheme is consistent. In fact, it is expected that in this case at a given interface point $\| \Psi^\bn_U(\bU^-, \bU^+)\|$ tends to zero as the mesh is refined, see e.g., the numerical experiments in~\cite{herty2023centrschemtwo}. Thus, the bounds in~\eqref{eq:couplingbound} are usually satisfied in practice.
  \item The coupling state of the fluid model with respect to the tangential velocities satisfies $[v_2]_L=v_2^+$, similarly as the tangential velocity state of the structure model in~\eqref{eq:RS-structure}.
  \item By construction our RS is idempotent. 
  \end{enumerate}
\end{remark}

\subsection{The numerical experiment}\label{sec:FSI-results}
In this Section we verify our approach in the two-dimensional experiment from~\cite{herty2018fluidstruccoupl}. 
We choose plastics as the linear elastic material with density $\rho = 1226$ kg ∕ $m^3$ and
Lamé constants $\mu = 1.4093 \times 10^9$ and $\lambda = 1.4093 \times 10^9$ N ∕ $m^2$. The fluid model accounts for air, an ideal gas, modeled by the constitutive law \eqref{eq:EoS} with parameters $\gamma = 1.4$ and $\pi = 0$ Pa.

The setup considers a high temperature gas bubble near a solid wall approximating a cavitation problem~\cite{dickopp2013coupl}. To this end we consider the solid domain $\Omega_1 = (-0.22, 0) \times (-0.22, 0.22)$~$m^2$ and the fluid domain $\Omega_2 = (0, 0.11) \times (-0.22, 0.22)$~$m^2$, which are separated by the interface $\Gamma=\{0\} \times (-0.22, 0.22)$~$m^2$. In the initial state zero deformation velocity and fluid velocity are assumed. Outside the bubble with center at $(0.02,0)$ and $15$ mm radius pressure $20 \times 10^6$ N / $m^2$ and temperature $293$ K are assumed, whereas inside the bubble pressure and temperature are $10^6$ N / $m^2$ and $693$ K, respectively. In $\Omega_1$ uniform states are assumed with $\sigma_{12} = \sigma_{22}=0$ and $\sigma_{11}$ taken such that the coupling condition~\eqref{eq:cplstress} holds. Similar to the numerical computation in~\cite{herty2018fluidstruccoupl} we discretize $\Omega_1$ over $4 \times 8$ and $\Omega_2$ over $2 \times 8$ cells at the base level and allow for 7 refinement levels with up to $512 \times 1024$ and $256 \times 1024$ cells within the AMR algorithm.
We compute the numerical solution using the DG-SSP-RK scheme developed in Section~\ref{sec:DG} together with the RS designed in Section~\ref{sec:FSI-RS}, for implementation details we refer to Section~\ref{sec:implementation}. The computation has been conducted in approximately four hours on an M4 MacBook Pro with 24 GB Ram using 8 MPI Threads computing in parallel. At the final time after 13,114 time steps 39,860 cells out of 786,432 have been active on the adaptive mesh.

\begin{figure}[h!] 
  \centering
  \includegraphics{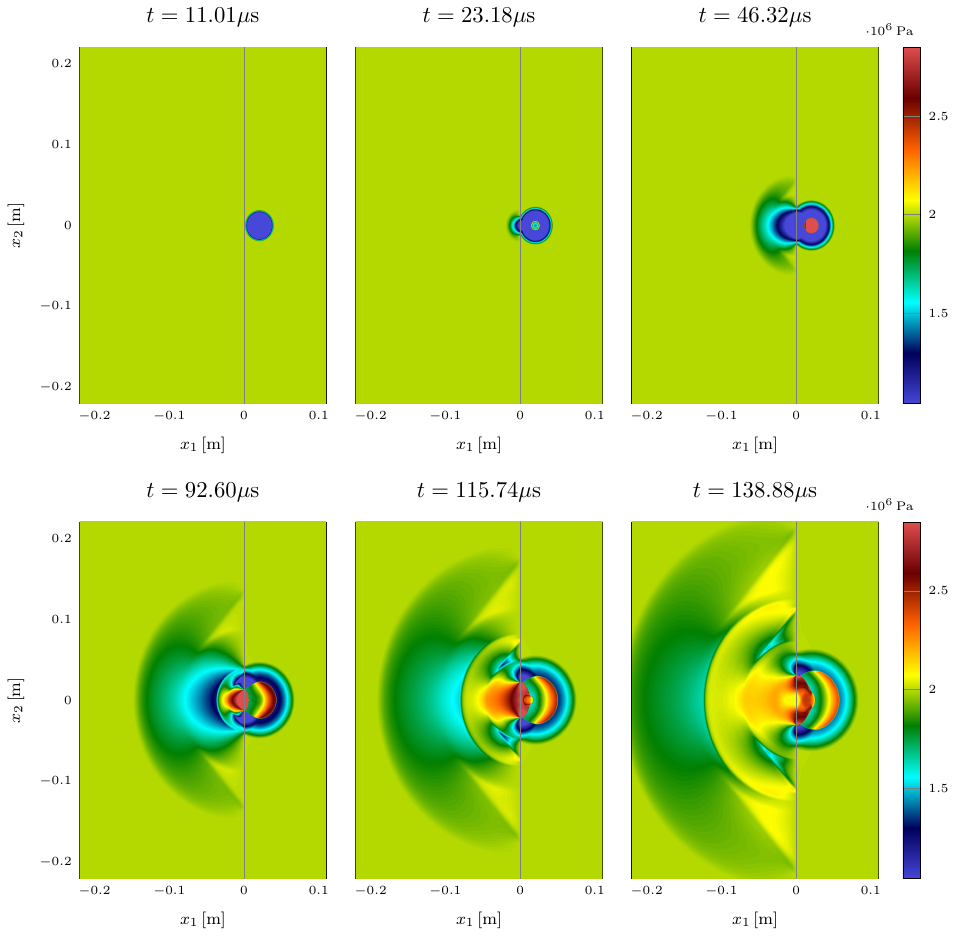}
    \caption{Numerical solution of the fluid-structure coupling problem in terms of negative stress $-\sigma_{11}$ in the solid domain ($x_1<0$\,m) and pressure $p$ in the fluid domain ($x_1>0$\,m) over six time instances.}\label{fig:fluidstructure}
\end{figure}

The numerical results shown in Figure~\ref{fig:fluidstructure} exhibit complex wave dynamics, including an expansion of the bubble, partial transmission and reflection at the interface as well as the formation of a dilatation and a shear wave in the solid material that travel with different velocities and are connected by a so-called von Schmidt wave. For a detailed discussion of the wave dynamics we refer to~\cite{herty2018fluidstruccoupl}, where the same dynamics have been captured. In fact, a comparison between the numerical solution based on our new approach and another numerical solution computed combining the SSP-DG scheme with a RS for the original nonlinear system solving nonlinear coupled half-Riemann problems has revealed identical wave structures. This validates our new relaxation based approach. We note that Figure~\ref{fig:fluidstructure} also shows that the scheme is capable to preserve the coupling condition~\eqref{eq:cplstress} on the discrete level over time, while resolving the solution around the interface in high detail.

\section{Conclusion}\label{sec:conclusion}
In this work we have presented a new numerical framework for conservation laws and hyperbolic systems coupled at
an interface in multiple space dimension. The relaxation approach that we follow avoids the need for information on the Lax curves of the coupled systems, which significantly simplifies its application to various problems as opposed to the established approach, in which two nonlinear coupled half-Riemann problems need to be solved. The multi-dimensional configuration is addressed by locally projecting the relaxed systems in normal direction to the interface. The obtained local relaxed coupling problem, albeit dependent on the direction of the projection, is similar to the one-dimensional problem discussed in~\cite{herty2023centrschemtwo}. Here we have taken the analysis of this problem one step further by providing a general construction algorithm for the RS at the coupling interface for arbitrary systems through the addition of further coupling conditions for the relaxation variable.

The proposed framework consisting of relaxation, projection and problem-specific RS gives rise to a numerical method; here we have followed a modal DG approach allowing for polynomial approximations of arbitrary order within each mesh cell. The construction of the scheme for the coupled relaxation system is straightforward and makes use of Godunov-fluxes for the linear hyperbolic systems and the proposed coupling RS at the interface. Using an implicit discretization of the relaxation source terms allows us to take the scheme to the discrete relaxation limit leading to a scheme for the original coupled problem. We have shown here that this asymptotic-preserving property holds for the class of DG-SSP$k$-IMEX schemes, each resulting in the fully explicit DG-SSP$k$ scheme in the relaxation limit.

Finally, we have applied our approach to the coupling of a linear elastic solid to a fluid and constructed a suitable RS in this way. The explicit form of the coupling states has allowed us to derive conditions under which the solution stays physically admissible. A numerical simulation of an experiment from literature has shown that our framework is capable of capturing complex wave phenomena in high resolution in the same quality as the established Lax curve based approach.  
The presented method will thus constitute a starting point to address further multi-dimensional coupled problems, e.g.,~from multiphase fluids and cavitation, in future work.
\bibliographystyle{abbrvurl}

\appendix

\section{Eigenstructure of the multidimensional relaxation system}\label{appx:diagoonalization}

In this appendix, details on the diagonalization of $\mathcal{A}_i^\bn$ in \eqref{eq:Ain} are discussed. Adopting the definition~\eqref{eq:Vhn} the diagonal eigenvalue matrix takes the form
\[
  \mathcal{D}_i^\bn =
  \begin{pmatrix}
    \bzero &  \\
           & \ddots & \\
           & & \bzero & \\
           & & & - \bSigma_i & \\
           & & & &  \bSigma_i
  \end{pmatrix}.
\]
Given $d-1$ vectors $\mathbf t_1, \dots, \mathbf t_d$ that form an orthonormal basis with $\bn$ as in Section \ref{sec:locrelaxationsystem} the right and left eigenmatrix are given by
\begin{align}
  \mathcal{R}_i^\bn &=
                      \begin{pmatrix}
                        \bzero     & \ldots & \bzero & \bI & \bI \\
                        t_{1,1} \, \bI  & \ldots &  t_{d-1,1} \, \bI & -n_1 \, \bSigma_i^{-1} (\bLambda_i^1)^2 & n_1 \, \bSigma^{-1} (\bLambda_i^1)^2\\
                        \vdots       & & \vdots  & \vdots & \vdots \\
                        t_{1,d} \, \bI  & \ldots &  t_{d-1,d} \, \bI &  -n_d \, \bSigma_i^{-1} (\bLambda_i^d)^2 & n_d \, \bSigma^{-1} (\bLambda_i^d)^2
                      \end{pmatrix},\\[8pt]
  (\mathcal{R}_i^\bn)^{-1} &=
                             \begin{pmatrix}
                               \bzero  &  t_{1,1}\,\bI +n_1 \, \mathbf E_i^1 & \ldots & t_{1,d}\, \bI +n_d \, \mathbf E_i^1 \\[5pt]
                               \vdots   &  \vdots            &        & \vdots \\[5pt]
                               \bzero & t_{d,1} \, \bI+n_1 \, \mathbf E_i^{d-1} & \ldots & t_{d,d} \, \bI + n_d \,\mathbf E_i^{d-1} \\[5pt]
                               \frac{1}{2} \, \bI & -\frac{1}{2} n_1 \, \bSigma^{-1}_i & \ldots & -\frac{1}{2} n_d \, \bSigma^{-1}_i \\[5pt]
                               \frac{1}{2} \, \bI & \frac{1}{2} n_1 \, \bSigma^{-1}_i  & \ldots & \frac{1}{2} n_d \, \bSigma^{-1}_i
                             \end{pmatrix}
\end{align}
using the sub-matrices
\[
  \mathbf E^k_i = - \bSigma^{-1}_i \left(\sum_{j=1}^d t_{k,j} n_j  (\bLambda_i^j)^2 \right) \bSigma^{-1}_i \quad \text{for } k \in \{1,\ldots, d-1\}.
\]
Due to the orthogonality assumption~\eqref{eq:samelambdai} implies $ \mathbf E^k_i = \bzero$ for all $k \in \{1,\ldots, d-1\}$.

\paragraph{Funding} The authors thank the Deutsche Forschungsgemeinschaft (DFG, German Research Foundation) for the financial support under 320021702/GRK2326 (Graduate College Energy, Entropy, and Dissipative Dynamics), through SPP 2410 (Hyperbolic Balance Laws in Fluid Mechanics: Complexity, Scales, Randomness) within the Project 525842915 and through SPP 2311 (Robust Coupling of Continuum-Biomechanical In Silico Models to Establish Active Biological System Models for Later Use in Clinical Applications -- Co-Design of Modelling, Numerics and Usability) within the Project 548864771.

\end{document}